\documentclass[11pt,twoside]{amsart}
\usepackage{amssymb}
\usepackage{graphicx,color}

\newtheorem{thm}{Theorem}[section]
\newtheorem{prop}[thm]{Proposition}

\newtheorem{defn}[thm]{Definition}

\newtheorem{ex}[thm]{Example}
\newtheorem{rem}[thm]{Remark}

\newcommand{\skipit}[1]{{}}
\newcommand{\prfend}{\hbox to7pt{\hfil}
\par\vskip-\baselineskip\hbox to\hsize
{\hfil\vbox {\hrule width6pt height6pt}}\vskip\baselineskip}

\newcommand{\ZZ}{\mathbb{Z}}
\newcommand{\RR}{\mathbb{R}}

\newcommand {\PP}{\mathbb{P}}

\DeclareMathOperator{\Image}{Im}

\DeclareMathOperator{\Proj}{Proj}

\newcommand{\A}{\tilde{A}}

\newcommand{\myarrow}[2]{\hbox to #1pt{\hfil$\to$\hfil}{\hskip-#1pt{\raise
10pt\hbox to#1pt{\hfil$\scriptscriptstyle #2$\hfil}}}}

\begin{document}

\title{On the classification of  Togliatti systems}

\author[Rosa M. Mir\'o-Roig]{Rosa M. Mir\'o-Roig}
\address{Department de matem\`{a}tiques i Inform\`{a}tica, Universitat de Barcelona, Gran Via de les Corts Catalanes 585, 08007 Barcelona,
Spain}
\email{miro@ub.edu}

\author[Mart\'{\i} Salat]{Mart\'{\i} Salat}
\address{Department de matem\`{a}tiques  i Inform\`{a}tica, Universitat de Barcelona, Gran Via de les Corts Catalanes 585, 08007 Barcelona,
Spain}
\email{msalatmo7@alumnes.ub.edu}

\begin{abstract} In \cite{MeMR}, Mezzetti and Mir\'{o}-Roig proved that the minimal number of generators $\mu (I)$ of  a minimal (smooth) monomial Togliatti system $I\subset k[x_{0},\dotsc,x_{n}]$ satisfies 
$2n+1\le \mu(I)\le \binom{n+d-1}{n-1}$ and they classify all smooth minimal monomial Togliatti systems $I\subset k[x_{0},\dotsc,x_{n}]$ with $2n+1\le \mu(I)\le 2n+2$.

In this paper, we address the first open case. We classify all smooth monomial Togliatti systems $I\subset k[x_{0},\dotsc,x_{n}]$ of forms of degree $d\ge 4$ with $\mu(I)=2n+3$ and $n\ge 2$ and all 
monomial Togliatti systems $I\subset k[x_0,x_1,x_2]$ of forms of degree $d\ge 6$ with $\mu(I)=7$.
\end{abstract}

\thanks{Acknowledgments:   The first  author was partially   supported
by  MTM2013-45075-P.
\\ {\it 2010 Mathematic Subject Classification.} 13E10,  14M25, 14N05,
14N15, 53A20.}

\maketitle

\tableofcontents

\markboth{R. M. Mir\'o-Roig, M. Salat}{On the classification of  Togliatti systems}

\today

%*****************************************************************************
\large

\section{Introduction}

The study and classification of smooth projective varieties satisfying at least one Laplace equation is a long standing problem in algebraic geometry. In \cite{MMO}, shedding new light on this subject, 
it was related to another long standing problem in commutative algebra: the study and classification of homogeneous artinian ideals failing the Weak Lefschetz Property (WLP). We contribute to these two 
problems resolving the first question that was left open  in \cite{MeMR}.

To be more precise.
Let $k$ be an algebraically closed field of characteristic $0$, $R=k[x_{0},\dotsc,x_{n}]$ and  $I=(F_{1},\dotsc, F_{s})\subset R$  a homogeneous artinian ideal generated by forms of the same degree $d$. 
Set $A=R/I$. We say that $A$ fails the WLP from degree $d-1$ to degree $d$ if the homomorphism $\times\ell:[A]_{d-1}\rightarrow[A]_{d}$ induced by a general linear form $\ell$ has not maximal
rank. As shown in \cite{MMO}, if $s\le \binom{n+d-1}{d}$ then this assertion is equivalent to saying that the projection $X_{n,d}^{I}$ of the $d$th Veronese variety $V(n,d)\subset\PP^{n}$ from $\langle F_{1},\dotsc, F_{s}\rangle$ 
satisfies at least one Laplace equation of order $d-1$. We call {\em Togliatti systems} the ideals satisfying these two equivalent statements (see Definition \ref{togliattisystem}). The name is in honour of E. Togliatti who gave a
complete classification of rational surfaces parameterized by cubics and satisfying at least one Laplace equation of order $2$ (see for instance \cite{T1},\cite{T2}).
Narrowing the field of study we deal only with monomial ideals $I$, so $X_{n,d}^{I}$ turns out to be a toric variety. In this sense, one can apply pure combinatoric tools due to Perkinson in \cite{P} to see whether $I$ is
a minimal monomial (smooth) Togliatti system (see Definition \ref{togliattisystem} and Propositions \ref{Hypersurface_Laplace} and \ref{smoothness}).
In \cite{MMO}, using these tools, Mezzetti, Mir\'o-Roig and Ottaviani classified all  smooth minimal monomial Togliatti systems of cubics in four variables and conjectured a further classification for $n\geq 3$.
By means of graph theory, this conjecture was proved by Mir\'o-Roig and Micha{\l}ek in \cite{MRM} where a classification of smooth minimal Togliatti systems $I\subset k[x_{0},\dotsc,x_{n}]$ of quadrics and cubics is achieved.
When $d\geq 4$ the picture becomes much more involved and a complete classification seems out of reach for now. Therefore, in \cite{MeMR} it was introduced another strategy: First to establish lower and upper bounds, depending on 
$n$ and $d\ge 2$, for the minimal number of generators of a monomial Togliatti system and then to study the monomial Togliatti systems with fixed number of generators.

In fact, in \cite{MeMR}, Mezzetti and Mir\'o-Roig bounded the number of generators of monomial Togliatti systems and classified all minimal monomial (smooth) Togliatti systems reaching the lower bound or close to reach it; namely 
those generated by $2n+1$ and $2n+2$ forms of degree $d\geq 4$. In this paper, we use again combinatoric tools and we  classify the first open case, i.e.,  all minimal monomial Togliatti systems generated by $2n+3$
forms of degree $d\geq 4$ and $n\geq 2$.

\vskip 2mm
Next we outline the structure of this note.  In
Section~\ref{defs and  prelim results} we  fix the notation
and we collect the basic results needed in the sequel. Then, in
Section~\ref{minimalnumbergenerators} we expose the main results of this note. Firstly, we give a complete classification of all minimal monomial Togliatti systems generated by $7$ forms of degree $d\geq 10$ in three variables
(see Theorem \ref{mainthm1}).
In order to achieve this classification we have had to consider all possible configurations of these $7$ monomials regarded in the integer standard simplex $d\Delta_{2}\subset\ZZ^{3}$ and then apply combinatorial criteria to each configuration.
Separating the problem in two basic cases which we have also had to separate into a few more subcases has helped so as to reduce the number of configurations to study. Once seen this classification we compute all minimal monomial Togliatti systems
generated by $7$ forms of degree $6\leq d\leq 9$ getting a complete scene of what occurs in three variables. From this result we can look apart all minimal monomial smooth Togliatti systems in three variables
generated by $7$ forms of degree $d\geq 6$ and close the open question we were dealing with.

\vskip 2mm \noindent {\bf Acknowledgement.} The first author of this paper warmly thanks Emilia Mezzetti for interesting conversations and many ideas developed in this paper.

%%%%%%%%%%%%%%%%%%%%%%%%%%%%%%%%%%%%%%%%%%%%%%%%%%
\section{Preliminaries} \label{defs and prelim results}

  We fix $k$  an algebraically closed field of characteristic zero, $R=k[x_{0},\dotsc,x_{n}]$ and $\PP^n=\Proj(k[x_{0},\dotsc,x_{n}])$.  Given a homogeneous artinian ideal 
  $I\subset k[x_{0},\dotsc,x_{n}]$, we  denote by $I^{-1}$ the ideal generated by the Macaulay inverse system of $I$ (see \cite[\S 3]{MMO} for details). 
  In this section we fix the notations and the main results that we use throughout this paper.
 In particular, we quickly
recall  the relationship between   the
existence of homogeneous  artinian ideals $I\subset
k[x_{0},\dotsc,x_{n}]$ failing the weak Lefschetz property; and  the
existence of (smooth) projective varieties $X\subset \PP^N$ satisfying
at least one Laplace equation of order $s\ge 2$.
For more details,  see \cite{MMO} and \cite{MRM}.

\begin{defn}\label{def of wlp}\rm Let $I\subset R$ be a
homogeneous artinian ideal. We  say that   $R/I$
 has the {\em weak Lefschetz property} (WLP, for short)
if there is a linear form $L \in (R/I)_1$ such that, for all
integers $j$, the multiplication map
\[
\times L: (R/I)_{j} \to (R/I)_{j+1}
\]
has maximal rank, i.e.\ it is injective or surjective.
We  often abuse notation and say that the ideal $I$ has the
WLP.    If for the general form $L \in (R/I)_1$
and for an integer number $j$
the map $\times L$ has not maximal rank we  say that
the ideal $I$ fails the WLP in degree $j$.
\end{defn}

  Though many algebras
 are expected to have the
WLP, establishing this property is often rather difficult. For
example, it was shown by R. Stanley \cite{st} and J. Watanabe
\cite{w} that a monomial artinian complete intersection ideal
$I\subset R$ has the WLP. By semicontinuity, it follows that
a {\em general}   artinian complete intersection ideal $I\subset
R$ has the WLP but it is open whether {\em every} artinian complete
intersection of height $\ge 4$
over a field of characteristic zero has the WLP. It is worthwhile to
point out that
the WLP of an artinian ideal $I$ strongly depends on the characteristic
of the ground field $k$ and,
 in positive characteristic, there are examples of artinian complete
intersection ideals $I\subset k[x_0,x_1,x_2]$ failing the WLP (see,
e.g. \cite[Remark 7.10]{MMN2}).

In \cite{MMO}, Mezzetti, Mir\'{o}-Roig, and Ottaviani showed  that  the
failure of the
WLP can be used to construct  (smooth) varieties satisfying at least
one   Laplace equation of order $s\ge 2$ (see also \cite{MRM}, \cite{MeMR} and \cite{MeMR2}).
 We have:

\begin{thm}\label{teathm} Let $I\subset R$ be an artinian
ideal
generated
by $r$ forms $F_1,...,F_{r}$ of degree $d$ with
$r\le \binom{n+d-1}{n-1}$. The following conditions are equivalent:
\begin{itemize}
\item[(1)] the ideal $I$ fails the WLP in degree $d-1$;
\item[(2)] the  homogeneous forms $F_1,...,F_{r}$ become
$k$-linearly dependent on a general hyperplane $H$ of $\PP^n$;
\item[(3)] the closure
$X:=\overline{\Image (
\varphi _{(I^{-1})_d})}\subset \PP^{\binom{n+d}{d}-r-1}$ of the image of
the rational map
$$\varphi _{(I^{-1})_d}:\PP^n \dashrightarrow \PP^{\binom{n+d}{d}-r-1}$$  associated to
  $(I^{-1})_d$  satisfies at least one Laplace equation of order
$d-1$.
\end{itemize}
\end{thm}

\begin{proof} See \cite[Theorem 3.2]{MMO}.
\end{proof}

The above result motivates the following definition:

\begin{defn} \label{togliattisystem} \rm Let $I\subset R$ be an artinian
ideal
generated
by $r$ forms $F_1,...,F_{r}$ of degree $d$, $r\le \binom{n+d-1}{n-1}$.
We say:
\begin{itemize}
\item[(i)] $I$  is a \emph{Togliatti system}
if it satisfies the three equivalent conditions in  Theorem \ref{teathm}.

\item[(ii)]   $I$ is a \emph{monomial Togliatti system} if,
in addition, $I$ (and hence $I^{-1}$) can be generated by monomials.

\item[(iii)]   $I$ is a \emph{smooth Togliatti system} if,
in addition, the $n$-dimensional   variety
$X$ is smooth.

\item[(iv)] A  monomial  Togliatti system $I$ is
\emph{minimal} if $I$ is generated by monomials $m_1, \dotsc ,m_r$ and
there is no proper subset $m_{i_1}, \dotsc ,m_{i_{r-1}}$ defining a
monomial  Togliatti system.
    \end{itemize}
\end{defn}

The names are in honor of Eugenio Togliatti who proved that for $n=2$ the only
smooth  Togliatti system of cubics is
$I=(x_0^3,x_1^3,x_2^3,x_0x_1x_2)\subset k[x_0,x_1,x_2]$ (see  \cite{T1}, \cite{T2}). This result has been reproved recently by Brenner and Kaid \cite{BK} in the context of weak Lefschetz property.
Indeed,  Togliatti  gave a classification of rational surfaces parameterized by cubics and satisfying at least one Laplace equation of order $2$: There is only one rational surface in $\PP^5$ parameterized by cubics and satisfying a Laplace equation of order 2; it is obtained from the $3$rd Veronese embedding $V(2,3)$ of $\PP^2$ by a suitable projection from four  points on it. In \cite{MMO}, the first author together with Mezzetti and Ottaviani
 classified all smooth rational 3-folds parameterized by cubics and satisfying a Laplace equation of order 2, and gave a conjecture to extend this result to varieties of higher dimension. This conjecture has been recently proved in \cite{MRM}. Indeed, the first author together with Micha{\l}ek  classified all smooth minimal Togliatti systems of quadrics and cubics. For $d\ge 4$, the picture becomes soon much more involved than in the case of quadrics and cubics, and for the moment a complete classification appears out of reach unless we introduce other invariants as, for example, the number of generators of $I$.

%%%%%%%%%%%%%%%%%%%%%%%%%%%%%%%%%%%%%%%%%%%%%%%%%%%%%%%%%%%%%%%%%%%%%%%%%%%%%%

%%%%%%%%%%%%%%%%%%%%%%%%%%%%%%%%%%%%%%%%%%%%%%%%%%

\section{The classification of Togliatti systems with $2n+3$ generators}
\label{minimalnumbergenerators}

From now on, we restrict our attention to  monomial artinian ideals
 $I\subset k[x_0,\ldots ,x_n]$, $n\ge 2$, generated by forms of degree $d\ge 4$. It is worthwhile
to recall that for monomial artinian ideals
to test the WLP there
is no need to consider a general linear form. In fact, we have

\begin{prop}
   \label{lem-L-element}
Let $I \subset R:=k[x_{0},\dotsc,x_{n}]$ be an artinian monomial ideal.
Then $R/I$ has the WLP if and only if  $x_0+x_1 + \dotsb + x_n$
is a Lefschetz element for $R/I$.
\end{prop}

\begin{proof} See \cite[Proposition 2.2]{MMN2}.
\end{proof}

Let $I\subset k[x_{0},\dotsc,x_{n}]$ be a minimal  monomial  Togliatti
systems   of forms of degree $d$ and
 denote by
$\mu (I)$ the minimal number of generators of $I$. In \cite{MM}, the first author and Mezzetti proved:

$$2n+1\le \mu(I)\le \binom{n+d-1}{n-1}.$$

\noindent In addition,   the Togliatti systems with number of generators reaching the lower bound
or close to the lower bound were classified. Indeed, we have

\begin{thm}
 Let $I\subset k[x_{0},\dotsc,x_{n}]$ be a minimal  monomial
Togliatti system of forms of degree $d\ge 4$.
 Assume that $\mu(I)=2n+1$. Then, up to
a permutation of the coordinates, one of the following cases holds:
 \begin{itemize}
 \item[(i)] $n\ge 2$ and $I=(x_1^d, \dotsc, x_n^d)+x_0^{d-1}(x_0,
\dotsc, x_n)$, or
 \item[(ii)] $(n,d)=(2,5)$ and
$I=(x_0^5,x_1^5,x_2^5,x_0^3x_1x_2,x_0x_1^2x_2^2)$, or
  \item[(iii)]  $(n,d)=(2,4)$ and
$I=(x_0^4,x_1^4,x_2^4,x_0x_1x_2^2,x_0^2x_1^2)$.
 \end{itemize}

  Furthermore, (i) and (ii) are smooth while  (iii) is not smooth.
\end{thm}
\begin{proof} See \cite[Theorem 3.7]{MeMR}.
\end{proof}

\begin{thm}
 Let $I\subset k[x_{0},\dotsc,x_{n}]$ be a smooth minimal  monomial
Togliatti system of forms of degree $d\ge 4$.
 Assume that $\mu(I)=2n+2$. Then, up to
a permutation of the coordinates, one of the following cases holds:
 \begin{itemize}
  \item[(i)] $n\ge 2$ and $I=(x_0^d, \dotsc, x_n^d)+m(x_0,
\dotsc, x_n)$ with $m=x_0^{i_0}x_1^{i_1}\dotsb x_n^{i_n}$ where $i_0\ge i_1\ge \dotsb \ge i_n\ge 0$, $i_2>0$ and $i_0+i_1+\dotsb +i_n=d-1$.
 \item[(ii)] $(n,d)=(2,5)$ and
$I=(x_0^5,x_1^5,x_2^5,x_0^3x_1x_2,x_0^2x_1^2x_2,x_0x_1^3x_2)$ or
$(x_0^5,x_1^5,x_2^5,x_0^3x_1x_2,$

\noindent$x_0x_1^3x_2,x_0x_1x_2^3)$ or $(x_0^5,x_1^5,x_2^5,x_0^2x_1^2x_2,x_0^2x_1x_2^2,x_0x_1^2x_2^2)$.
  \item[(iii)]  $(n,d)=(2,7)$ and
  $d=7$ and
$I=(x_0^7,x_1^7,x_2^7,x_0^3x_1^3x_2,x_0^3x_1x_2^3,x_0x_1^3x_2^3)$ or
$(x_0^7,x_1^7,x_2^7,$

\noindent$x_0^5x_1x_2,x_0x_1^5x_2,x_0x_1x_2^5)$,
$(x_0^7,x_1^7,x_2^7,x_0x_1x_2^5,x_0^3x_1^3x_2,x_0^2x_1^2x_2^3)$ or
$(x_0^7,x_1^7,x_2^7,x_0^4x_1x_2^2,$

\noindent$x_0^2x_1^4x_2,x_0x_1^2x_2^4)$.
 \end{itemize}

\end{thm}
\begin{proof} See \cite[Theorem 3.17 and Proposition 3.19]{MeMR}.
\end{proof}

In this paper, we address the first open case and we classify all   smooth
minimal monomial Togliatti systems $I\subset k[x_{0},\dotsc,x_{n}]$ of
forms of degree $d\ge 4$  with $\mu(I)=2n+3$ (see Theorem \ref{mainthm2}) as well as  all minimal monomial Togliatti systems $I\subset k[x_0,x_1,x_2]$ of
forms of degree $d\ge 6$ with $\mu(I)=7$ (see Theorem \ref{mainthm1}).

\vskip 4mm
In order to achieve this classification, we  associate to any artinian monomial ideal a polytope
and
we  tackle our problem with tools  coming from
combinatorics. In fact,
  the failure of the WLP of an artinian monomial ideal $I\subset k[x_{0},\dotsc,x_{n}]$ can be
established by purely  combinatoric properties of the associated
polytope $P_I$. To state this result we need to  fix some extra notation.

\vskip 2mm
Given an artinian monomial ideal
 $I\subset k[x_{0},\dotsc,x_{n}]$ generated by monomials  of degree $d$ and  its inverse
system $I^{-1}$, we denote by $\Delta _n$ the standard $n$-dimensional simplex
in the lattice $\ZZ^{n+1}$, we consider $d\Delta _n$ and we define the
polytope $P_I$ as the convex hull of the finite subset $A_I\subset
\ZZ^{n+1}$ corresponding to monomials of degree $d$  in $I^{-1}$. As
usual we define:
$$\text{ Aff}_{\ZZ}(A_I):=\{ \sum _{x\in A_I} n_x\cdot x \mid n_x\in
\ZZ, \quad \sum _{x\in A_I} n_x=1 \}$$
the sublattice $\text{ Aff}_{\ZZ}(A_I)$ in $\ZZ^{n+1}$
generated by $A_I$. We have the following criterion which will play an important role in the proof of our main result.

\begin{prop} \label{Hypersurface_Laplace}
Let $I\subset k[x_{0},\dotsc,x_{n}]$ be an artinian monomial ideal
generated by monomials of degree $d$. Assume $r\le \binom{n+d-1}{n-1}$.
Then, $I$ is a Togliatti system if and only if there exists a
hypersurface of degree $d-1$ containing  $A_I\subset \ZZ^{n+1}$.
 In addition, $I$ is a minimal Togliatti system
  if and only if any such hypersurface $F$  does not contain any
integral point of $d\Delta_n\setminus A_I$ except possibly  some of the
vertices of $d\Delta_n$.
\end{prop}
\begin{proof} This follows from Theorem \ref{teathm} and \cite[Proposition 1.1]{P}.
\end{proof}

\begin{ex} \rm
The artinian monomial ideal $$I=(x_0,x_1)^3+(x_2,x_3)^3+(x_4^3,x_0x_2x_4,x_0x_3x_4,x_1x_2x_4,x_1x_3x_4)\subset
k[x_0,x_1,x_2,x_3,x_4]$$  defines a minimal monomial Togliatti system of
cubics. In fact, the set  $A_I\subset \ZZ^5$ is:
{\small $$A_I=\{ (2,0,1,0,0), (1,0,2,0,0),(2,0,0,1,0),(1,0,0,2,0),(2,0,0,0,1),(1,0,0,0,2),
 (0,2,1,0,0),$$
$$
(0,1,2,0,0),
(0,2,0,1,0),(0,1,0,2,0),(0,2,0,0,1),(0,1,0,0,2),(0,0,2,0,1),(0,0,1,0,2),(0,0,0,2,1),$$
$$(0,0,0,1,2),
(1,1,1,0,0),(1,1,0,1,0),(1,1,0,0,1),(1,0,1,1,0),(0,1,1,1,0),(0,0,1,1,1)\}.$$}
There is a quadric,  and only one, containing all points of
$A_I$ and no integral point of $4\Delta_4\setminus A_I$, namely,
$$Q(x_0,x_1,x_2,x_3,x_4)=2\sum _{i=0} ^4x_{i}^2+9(x_0x_1+x_2x_3)-5\sum _{0\le i<j\le 4}x_{i}x_{j}.$$
\end{ex}

The following  criterion allows us to check if a subset $A$ of
points in the lattice $\ZZ^{n+1}$ defines a smooth toric variety $X_A$
or not.

\begin{prop} \label{smoothness}
 Let $I\subset k[x_{0},\dotsc,x_{n}]$ be an artinian monomial ideal
generated by $r$ monomials  of degree $d$. Let $A_I\subset \ZZ^{n+1}$ be the
set of integral points corresponding to monomials in $(I^{-1})_d$, $S_I$
the semigroup generated by $A_I$ and 0, $P_I$ the convex hull of $A_I$
and $X_{A_I}$ the projective toric variety associated to the polytope $P_I$. Then
 $X_{A_I}$ is smooth if and only if for any non-empty face $\Gamma $ of
$P_I$ the following conditions hold:

 (i) The semigroup $S_I/\Gamma$ is isomorphic to $\ZZ^m_+$ with
$m=\dim(P_I)-\dim \Gamma +1$.

 (ii) The lattices $\ZZ^{n+1}\cap \text{ Aff}_{\RR}(\Gamma )$ and
$\text{ Aff}_{\ZZ}(A_I\cap \Gamma)$ coincide.
\end{prop}
\begin{proof} See \cite[Chapter 5, Corollary 3.2]{GKZ}. Note that in
this case  $X_{A_I}=X$ where $X$ is the closure of the image of the rational map $\varphi _{[I^{-1}]_d}:\PP^n \longrightarrow\PP^{\binom{n+d}{d}-r-1}$.
\end{proof}

\vskip 2mm
As a direct application of this criterion we get:

\begin{ex} \rm
Let
 $$I=(x_0,x_1)^3+(x_2,x_3)^3+(x_4^3,x_0x_2x_4,x_0x_3x_4,x_1x_2x_4,x_1x_3x_4)\subset
k[x_0,x_1,x_2,x_3,x_4]$$  be the  minimal monomial Togliatti system of
cubics described in Example 3.5. Applying the above smoothness criterion we get that the toric variety $X_{A_I}$ is smooth.
\end{ex}

For any integer $d\ge 3$, we define $M(d):=\{ x_0^ax_1^bx_2^c \mid a+b+c=d \text{ and } a,b,c\le d-1\}$ and we consider the sets of ideals:

{\small $$A=\{(x_0^2x_2,x_0x_1^2,x_1^3,x_1x_2^2),(x_0^2x_2,x_0x_1x_2,x_1^3,x_1^2x_2),(x_0^2x_2,x_0x_1x_2,x_1^3,x_1x_2^2), (x_0^2x_1,x_0x_2^2,x_1^3,x_1^2x_2),$$
$$(x_0^2x_1,x_0x_2^2,x_1^3,x_1x_2^2), (x_0^2x_2,x_0x_2^2,x_1^3,x_1x_2^2), (x_0x_1^2,x_0x_2^2,x_1^3,x_1x_2^2),(x_0^2x_2,x_1^2x_2,x_1^3,x_2^3)  ,(x_0x_2^2,x_1^2x_2,x_1^3,x_2^3),$$
$$ (x_0^2x_2,x_0x_1^2,x_1^2x_2,x_1x_2^2),\! (x_0x_1^2,x_0x_2^2,x_1^2x_2,x_1x_2^2),\!(x_0^2x_1,x_0x_1^2,x_1^3,x_2^3),\!(x_0^2x_1,x_1^2x_2,x_1^3,x_2^3),\! (x_1^2x_2,x_1x_2^2,x_1^3,x_2^3),$$
$$ (x_0x_2^2,x_1^2x_2,x_1x_2^2,x_1^3),\!(x_0x_1x_2,x_0x_2^2,x_1^3,x_1x_2^2),\!(x_0^2x_2,x_0x_2^2,x_1^3,x_1^2x_2),\!(x_0x_1^2,x_0x_2^2,x_1^3,x_2^3),(x_0^2x_2,x_0x_1^2,x_1^3,x_2^3),$$
$$ (x_0x_1x_2,x_0x_2^2,x_1^3,x_1^2x_2),(x_0^2x_2,x_1^2x_2,x_1x_2^2,x_1^3) \},$$}
{\small $$B=\{(x_0^2x_1x_2,x_0x_2^3,x_1^4,x_1^3x_2),(x_0^3x_2,x_0x_1^2,x_2,x_1^4,x_1x_2^3),(x_0^2x_2^2,x_0x_1^2x_2,x_1^4,x_2^4),(x_0^2x_1x_2,x_1^2x_2^2,x_1^4,x_2^4)\}$$}
\noindent and
$$C=\{(x_0^2x_1x_2^2,x_0x_1^3x_2,x_1^5,x_2^5),(x_0^3x_1x_2,x_0x_1^2x_2^2,x_1^5,x_2^5)\}.$$

\begin{thm}\label{mainthm1}
Let $I\subset k[x_0,x_1,x_2]$ be a minimal  monomial
Togliatti system of forms of degree $d\ge 10$.
 Assume that $\mu(I)=7$. Then,  up to
a permutation of the coordinates, one of the following cases holds:
 \begin{itemize}
 \item[(1)] $I=(x_0^d,x_1^d, x_2^d)+m(x_0^2,x_1^2,x_0x_2,x_1x_2)$ with $m\in M(d-2)$, or
 \item[(2)] $I=(x_0^d,x_1^d, x_2^d)+m(x_0^2,x_1^2,x_0x_1,x_2^2)$ with $m\in M(d-2)$, or
 \item[(3)]  $I=(x_0^d,x_1^d, x_2^d)+m(x_0^3,x_1^3,x_2^3,x_0x_1x_2)$ with $m\in M(d-3)$, or
 \item[(4)] $I=(x_0^d,x_1^d, x_2^d)+x_0^{d-3}J $ with $J\in A$, or
 \item[(5)] $I=(x_0^d,x_1^d, x_2^d)+x_0^{d-4}J $ with $J\in B$, or
 \item[(6)] $I=(x_0^d,x_1^d, x_2^d)+x_0^{d-5}J $ with $J\in C$.
 \end{itemize}
\end{thm}

\begin{proof} It is easy to check that all of these ideals are minimal Togliatti systems. Vice versa, let  us write
$I=(x_{0}^{d},x_{1}^{d},x_{2}^{d},m_{1},m_{2},m_{3},m_{4})$ where for $1\leq i\leq 4,\;m_{i}=x_{0}^{a_{i}}x_{1}^{b_{i}}x_{2}^{c_{i}}$ with $a_{i}+b_{i}+c_{i}=d$.
We consider $A_{I}\subset d\Delta_{2}\cap\ZZ^{3}$ and we slice $A_{I}$ with planes in three possible manners:

For $0\leq j\leq 2$ and $0\leq i\leq d$, we define $H_{i}^{j}:=\{(t_{0},t_{1},t_{2})|t_{j}=i\}$ and $A_{I}^{(i,j)}:=A_{I}\cap H_{i}^{j}$.

We  divide the proof in two cases:

\vskip 4mm
\hspace{5mm}\underline{\textsc{Case $1$:}} There exist $1\leq i_{a},\;i_{b},\;i_{c}\leq 4$ such that $a_{i_{a}},\,b_{i_{b}},\,c_{i_{c}}\leq1$.

\hspace{5mm}\underline{\textsc{Case $2$:}} There exists one variable whose square divides all monomials $m_{i}$.

\vskip 4mm
\noindent\underline{\textsc{Case $1$:}} None of the squares of the variables divide the four monomials $m_{1},\,m_{2},\,m_{3}$ and $m_{4}$.
Up to permutation of the variables, we have two possibilities:

\vskip 2mm
\textsc{Case 1A:} $I=(x_{0}^{d},x_{1}^{d},x_{2}^{d},x_{0}^{e_{1}}x_{1}^{a}x_{2}^{d-a-e_{1}}, x_{0}^{b}x_{1}^{e_{2}}x_{2}^{d-b-e_{2}}, x_{0}^{c}x_{1}^{d-c-e_{3}}x_{2}^{e_{3}}, x_{0}^{\alpha}x_{1}^{\beta}x_{2}^{d-\alpha-\beta})$
with $0\leq e_{1},e_{2},e_{3}\leq 1$, $d-2-e_{1}\geq a\geq 2$, $d-2-e_{2}\geq b\geq 2$, $d-2-e_{3}\geq c\geq 2$ and only one of the exponents $\alpha, \beta, d-\alpha-\beta$ is $\leq 1$.

\textsc{Case 1B:} $I=(x_{0}^{d},x_{1}^{d},x_{2}^{d},x_{0}^{e_{1}}x_{1}^{e_{2}}x_{2}^{d-e_{1}-e_{2}}, x_{0}^{a}x_{1}^{d-a-e_{3}}x_{2}^{e_{3}},x_{0}^{\alpha}x_{1}^{\beta}x_{2}^{d-\alpha-\beta}, 
x_{0}^{\gamma}x_{1}^{\delta}x_{2}^{d-\gamma-\delta})$ with $0\leq e_{1},e_{2},e_{3}\leq 1$.

\vskip 2mm
In both cases, a straightforward computation using the hypothesis $d\ge 10$ shows that when we restrict to $x_0+x_1+x_2$ the $7$ monomials remain $k-$linearly independents. Therefore, $I$ is not a Togliatti system.

\vspace{4mm}
\noindent\underline{\textsc{Case $2$:}} Without loss of generality we can suppose that $x_{0}^{2}$ divides each monomial $m_{i}$. We can also assume that $a_{1}\geq a_{2}\geq a_{3}\geq a_{4}=s\geq 2$.

Let  $F_{d-1}$ be a plane curve of degree $d-1$
containing all integral points of $A_{I}$. Since $s\geq 2$, it factorizes as $F_{d-1}=L_{0}^{0}L_{1}^{0}\dotsb L_{s-1}^{0}F_{d-s-1}$. Indeed,  $F_{d-1}$ has degree $d-1$
and contains the $d$ points in $A_{I}^{(1,0)}$. So, $F_{d-1}=L_{1}^{0}F_{d-2}$. Since $F_{d-2}$ contains all $d-1$ points of $A_{I}^{(0,0)}$ it factorizes as $F_{d-1}=L_{0}^{0}L_{1}^{0}
F_{d-3}$. Repeating the process we arrive to $F_{d-1}=L_{0}^{0}L_{1}^{0}\dotsb L_{s-1}^{0}F_{d-s-1}$.
We claim that $a_{3}=a_{4}=s\geq 2$.  If $a_{3}>a_{4}=s$ then
$A_{I}^{(s,0)}$ has $d-s$ points and $F_{d-s-1}$  contains them. Hence,  $F_{d-s-1}=L_{s}^{0}F_{d-s-2}$ contradicting the minimality of $I$ (Proposition \ref{Hypersurface_Laplace}).

So far we have $a_3=a_4=s\ge 2$ and  $F_{d-1}=L_{0}^{0}\dotsb L_{s-1}^{0}F_{d-s-1}$ where $F_{d-s-1}$ is a plane curve of degree $d-s-1$ which contains all integer points of $
\tilde{A}_{I}:=A_{I}\setminus\left(\cup_{k=0}^{s-1}A_{I}^{(k,0)}\right)$. Set $\tilde{A}_{I}^{(i,j)}=\tilde{A}_{I}\cap H_{i}^{j}$. We  distinguish four subcases:

\vskip 2mm

\textsc{Case 2A:} $a_{1}=a_{2}=a_{3}=a_{4}=:s\geq 2$.

\textsc{Case 2B:} $u:=a_{1}>a_{2}=a_{3}=a_{4}=:s\geq 2$.

\textsc{Case 2C:} $u:=a_{1}=a_{2}>a_{3}=a_{4}=:s\geq 2$.

\textsc{Case 2D:} $u:=a_{1}>v:=a_{2}>a_{3}=a_{4}=:s\geq 2$.

\vskip 4mm
\noindent \textsc{Case 2A:} We assume $a_{1}=a_{2}=a_{3}=a_{4}=:s\geq 2$. In this case, $F_{d-1}=L_{0}^{0}\dotsb L_{s-1}^{0}L_{s+1}^{0}\dotsb$ $L_{d-1}^{0}$. Therefore,
$s=d-3$ and $I=(x_{0}^{d},x_{1}^{d},x_{2}^{d})+x_{0}^{d-3}(x_{1}^{3},x_{1}^{2}x_{2},$ $x_{1}x_{2}^{2},x_{2}^{3})$, which is of type $(4)$.

\vskip 4mm
\noindent\textsc{Case 2B:} We assume $u:=a_{1}>a_{2}=a_{3}=a_{4}=:s\geq 2$. In this case, $u\le s+2$. Indeed, if $u>s+2$ we have, $F_{d-s-1}=L_{s+1}^0\dotsb L_{u-1}^0F_{d-u}$ with $F_{d-u}$ a plane
curve of degree $d-u$ which contains in particular $A_{I}^{(s,0)}$. By minimality, $\#(F_{d-u}\cap A_{I}^{(s,0)})=d-s-2>d-u$ (Proposition \ref{Hypersurface_Laplace}) and we have $F_{d-u}=L_{s}^{0}F_{d-u-1}$, which
is a contradiction. Then, up to permutation of variables, $I$ is as one of the following cases:

\vskip 2mm
\textsc{Case} b1: $u=s+1$ and $I=(x_{0}^{d},x_{1}^{d},x_{2}^{d})+x_{0}^{s}(x_{0}x_{1}^{a}x_{2}^{d-a-s-1},x_{1}^{b}x_{2}^{d-b-s},x_{1}^{c}x_{2}^{d-c-s},x_{1}^{e}x_{2}^{d-e-s})$.

\textsc{Case} b2: $u=s+2$ and $I=(x_{0}^{d},x_{1}^{d},x_{2}^{d})+x_{0}^{s}(x_{0}^{2}x_{1}^{a}x_{2}^{d-a-s-2},x_{1}^{b}x_{2}^{d-b-s},x_{1}^{c}x_{2}^{d-c-s},x_{1}^{e}x_{2}^{d-e-s})$.

\vskip 2mm

\noindent\textsc{Case} b1: In this case we are removing three points from $H_{s}^{0}$ and one from $H_{s+1}^{0}$.
Up to permutation of the variables $y$ and $z$, we can assume  $d-s-1\geq a\geq \lfloor\frac{d-s-1}{2}\rfloor$ and  $d-s\geq b>c>e\geq 0$.
Let us first  suppose that $\lfloor\frac{d-s-1}{2}\rfloor>e\geq 0$.
In this case
$$\#(F_{d-s-1}\cap \tilde{A}_{I}^{(0,1)})=
\left\{\begin{array}{ll}
d-s&e\geq 1\\
d-s-1&e=0.
\end{array}\right.$$

If $e\geq 1$, then $\#\tilde{A}_{I}^{(0,1)}=d-s$, $F_{d-s-1}=L_{0}^{1}F_{d-s-2}=L_{0}^{1}\dotsb L_{e-1}^{1}F_{d-s-e-1}$ and $F_{d-s-e-1}$ contains the integer points of $\tilde{A}_{I}^{(e,1)}$. Since $a>e$ and $b>c>e$, we have $\#\tilde{A}_{I}^{(e,1)}=d-s-e$ and  $F_{d-s-e-1}=L_{e}^{1}F_{d-s-e-2}$ contradicting the minimality of $I$.
Therefore it must be $e=0$, and $m_{4}=x_{0}^{s}x_{2}^{d-s}$ with $d-s-1\geq c\geq 1$.
Let us consider
$$\#(F_{d-s-1}\cap \tilde{A}_{I}^{(1,1)})=
\left\{\begin{array}{ll}
        d-s& a,c\geq 2\\
        d-s-1& a=1,\;c\geq 2\\
        d-s-1& a\geq 2,\;c=1\\
        d-s-2& a=c=1
       \end{array}\right. $$ and we  study the four possibilities.

If $a,c\geq 2$ then we have the factorization $F_{d-s-1}=L_{1}^{1}F_{d-s-2}$. In particular, $F_{d-s-2}$ is a plane curve of degree $d-s-2$ containing the $d-s-1$ points of $\tilde{A}_{I}^{(0,1)}$. So, $F_{d-s-2}=L_{0}^{1}F_{d-s-3}$ which contradicts  the minimality of $I$. Therefore, if $a\geq 2$, then $c=1$.

If $a=1$, we have $ d-2\geq s\geq d-3$. If $s=d-2$, then $c=1$ and $I=(x_{0}^{d},x_{1}^{d},x_{2}^{d},x_{0}^{d-1}x_{1},$ $x_{0}^{d-2}x_{1}^{2},x_{0}^{d-2}x_{1}x_{2},x_{0}^{d-2}x_{2}^{2})$
which is not a Togliatti system. Otherwise, $s=d-3$, then we have several possibilities:

\noindent (i) $c\geq 2$ and $I=(x_{0}^d,x_{1}^d,x_{2}^d)+x_{0}^{d-3}(x_{1}^3,x_{2}^3,x_{0}x_{1}x_{2})+(x_{0}^{d-3}x_{1}^{2}x_{2})$ which is not minimal.

\noindent (ii) $c=1$ and  $I=(x_{0}^{d},x_{1}^{d},x_{2}^{d})+x_{0}^{d-3}(x_{0}x_{1}x_{2},x_{1}^{3},x_{2}^{3})+(x_{0}^{d-3}x_{1}x_{2}^{2})$ or $I=(x_{0}^{d},x_{1}^{d},x_{2}^{d})+x_{0}^{d-3}x_{1}x_{2}(x_0,x_1,x_2)+(x_{0}^{d-3}x_{2}^{3})$. Both of them are not minimal.

If $d-s-1\geq a\geq 2$ and $s\leq d-3$. We have  $e=0$, $c=1$ and  $(m_{1},m_{2},m_{3},m_{4})=(x_{0}^{s+1}x_{1}^{a}x_{2}^{d-a-s-1},x_{0}^{s}x_{1}^{b}x_{2}^{d-b-s},$ $x_{0}^{s}x_{1}x_{2}^{d-s-1},x_{0}^{s}x_{2}^{d-s})$ with
$d-s\geq b\geq 2$. Let us consider
$$\#(F_{d-s-1}\cap \tilde{A}_{I}^{(0,2)})=
\left\{
\begin{array}{ll}
 d-s&d-s-1\geq b,\;d-s-2\geq a\\
 d-s-1&a=d-s-1,\;d-s-1\geq b\\
 d-s-1&b=d-s,\;d-s-2\geq a\\
 d-s-2&b=d-s,\;a=d-s-1.
\end{array}\right.$$

In the first case, we have the factorization $F_{d-s-1}=L_{0}^{2}F_{d-s-2}=L_0^2L_{1}^{0}F_{d-s-3}$ and
it contradicts the minimality of $I$.

If $a=d-s-1$ and $d-s-1\geq b$, then  $b=d-s-1$. Otherwise, we would have  $F_{d-s-1}=L_{1}^{2}F_{d-s-2}$ and it would contradict the minimality of $I$.
Therefore we have $I=(x_{0}^{d},x_{1}^{d},x_{2}^{d})+x_{0}^{s}(x_{0}x_{1}^{d-s-1},x_{1}^{d-s-1}x_{2},x_{1}x_{2}^{d-s-1},x_{2}^{d-s})$ with $s\leq d-3$. For $s=d-3$ it corresponds to a Togliatti system of type (4),
while for $s\leq d-4$ is not Togliatti because when we restrict to $x_0+x_1+x_2=0$ the generators, they remain $k-$linearly independent.

If $d-s-2\geq a$ and $b=d-s$,  then  $a=d-s-2$.
Hence we have $s\leq d-4$ and $I=(x_{0}^{d},x_{1}^{d},x_{2}^{d})+x_{0}^{s}(x_{0}x_{1}^{d-s-2}x_{2},x_{1}^{d-s},x_{1}x_{2}^{d-s-1},$ $x_{2}^{d-s})$ which is never a Togliatti system.

Finally, if $b=d-s$ and $a=d-s-1$, then $s\leq d-3$ and $I=(x_{0}^{d},x_{1}^{d},x_{2}^{d})+x_{0}^{s}(x_{0}x_{1}^{d-s-1},x_{1}^{d-s},x_{1}x_{2}^{d-s-1},x_{2}^{d-s})$ which is a Togliatti system of type $(4)$ for $s=d-3$ while for $s\leq d-4$ it is
not Togliatti.

\vskip 2mm
To finish with the case b1, we have to see what happens when $d-s-2\geq e\geq\lfloor\frac{d-s-1}{2}\rfloor$. In this case $s\leq d-3$. Let us see that $a=e$. Otherwise, we can suppose $a>e$ (the other case is
symmetric) and we have the factorization $F_{d-s-1}=L_{0}^{1}\dotsb L_{e-1}^{1}F_{d-s-e-1}$. Since $a>e$ and $b>c>e$, $\tilde{A}_{I}^{(e,1)}$ has $d-s-e$ points and we have the factorization $F_{d-s-e-1}=L_{e}^{1}F_{d-s-
e-2}$ which contradicts the minimality of $I$. Hence $a=e$ and in particular $d-s-1>a$ and $d-s\geq b>c>a$.

Let us consider $\tilde{\tilde{A}}_{I}:=\tilde{A}_{I}\setminus\left(\cup_{k=0}^{a-1}\right)$ in the same spirit as $A_{I}$ and $\tilde{A}_{I}$. If $b=d-s$, then $\tilde{\tilde{A}}_{I}^{(0,2)}$ consists in $d-s-e$ different points.
Otherwise, $d-s-1\geq b$ and $\#\tilde{\tilde{A}}_{I}^{(0,2)}=d-s+1-e$. In both cases $\tilde{\tilde{A}}_{I}^{(0,2)}$ have more points than the degree of the curve $F_{d-s-e-1}$ which passes through them.
Therefore,
$F_{d-s-e-1}=L_{0}^{2}F_{d-s-e-2}$ and $d-s-1\geq b$. Since $m_{2}$ cannot be aligned vertically with any other monomial $m_{i}$, we can repeat this argument until we get that $b=c+1$ and
$F_{d-s-e-1}=L_{0}^{2}\dotsb L_{b-1}^{2}F_{d-s-e-b-1}$. Now $F_{d-s-e-b-1}$ is a plane curve of degree $d-s-e-b-1$ containing all $d-s-e-b$ points of $\tilde{\tilde{A}}_{I}^{(b,2)}$. Hence, we can factorize
$F_{d-s-e-b-1}=L_{b}^{2}F_{d-s-e-b-2}$ contradicting the minimality assumption.

\vskip 2mm
\noindent\textsc{Case} b2: We are removing from $d\Delta_{2}$ to get $\tilde{A}_{I}$: three points of $H_{s}^{0}$ and  one from $H_{s+2}^{0}$.
Up to permutation of the variables $y$ and $z$, we can suppose that $d-s-2\geq a\geq \lfloor\frac{d-s-2}{2}\rfloor$ and  $d-s\geq b>c>e\geq 0$.

Let us suppose first that $\lfloor\frac{d-s-2}{2}\rfloor>e\geq 0$. We argue as in the case $u=s+1$ to prove that $e=0$. Let us consider
$\#(F_{d-s-1}\cap \tilde{A}_{I}^{(1,1)})$. Using the same argumentation we prove that if $a,c\geq 2$ we get a contradiction. If $a=1$, then either $s=d-3$ or $s=d-4$ and we have the following cases:

\noindent(i) If $s=d-3$, then $(m_{1},m_{2},m_{3},m_{4})$ is $x_{0}^{d-3}(x_{0}^{2}x_{1},x_{1}^{3},x_{1}^{2}x_{2},x_{2}^{3})$, $x_{0}^{d-3}(x_{0}^{2}x_{1},x_{1}^{3},x_{1}x_{2}^{2},x_{2}^{3})$ or
$x_{0}^{d-3}(x_{0}^{2}x_{1},x_{1}^{2}x_{2},x_{1}x_{2}^{2},x_{2}^{3})$. All of them are Togliatti systems of type $(4)$.

\noindent(ii) If $s=d-4$, then $(m_{1},m_{2},m_{3},m_{4})$ is $x_{0}^{d-4}(x_{0}^{2}x_{1}x_{2}, x_{1}^{4},x_{1}^{3}x_{2},x_{2}^{4})$, $x_{0}^{d-4}(x_{0}^{2}x_{1}x_{2}, x_{1}^{4}, x_{1}^{2}x_{2}^{2}, x_{2}^{4})$,
$x_{0}^{d-4}(x_{0}^{2}x_{1}x_{2}, x_{1}^{4},x_{1}x_{2}^{3},x_{2}^{4})$, $x_{0}^{d-4}(x_{0}^{2}x_{1}x_{2}, x_{1}^{3}x_{2},x_{1}^{2}x_{2}^{2},x_{2}^{4})$,
$x_{0}^{d-4}(x_{0}^{2}x_{1}x_{2}, x_{1}^{3}x_{2},x_{1}x_{2}^{3}, x_{2}^{4})$ or $x_{0}^{d-4}($ $x_{0}^{2}x_{1}x_{2}, x_{1}^{2}x_{2}^{2},$ $x_{1}x_{2}^{3},x_{2}^{4})$. The only one  which is
a minimal Togliatti system is the second one and it is of type $(5)$.

Now, we assume $e=0$, $c=1$ and $d-s-2\geq a\geq 2$. In particular, $s\leq d-4$. We  consider  $\#(F_{d-s-1}\cap \tilde{A}_{I}^{(0,2)})$, and see
that if $d-s-1\geq b\geq 2$ and $d-s-3\geq a\geq 2$, there is a contradiction with the minimality of $I$. If $b=d-s$ and $a\le d-s-3$  (resp. $a=d-s-2$ and $b\le d-s-1 $) then
$a=d-s-3$ (resp. $b=d-s-1$). Otherwise we would incur again to a contradiction with the minimality of $I$. So, we have three possibilities.

\noindent(i) $a=d-s-3\geq 2$, $b=d-s$, $s\leq d-5$ and $I=(x_{0}^{d},x_{1}^{d},x_{2}^{d})+x_{0}^{s}(x_{0}^{2}x_{1}^{d-s-3}x_{2},x_{1}^{d-s},$ $x_{1}x_{2}^{d-s-1},x_{2}^{d-s})$.

\noindent(ii) $a=d-s-2$, $b=d-s-1$, $s\leq d-4$ and $I=(x_{0}^{d},x_{1}^{d},x_{2}^{d})+x_{0}^{s}(x_{0}^{2}x_{1}^{d-s-2},x_{1}^{d-s-1}x_{2},$ $x_{1}x_{2}^{d-s-1},x_{2}^{d-s})$.

\noindent(iii) $a=d-s-2$, $b\!=\!d-s$, $s\leq d-4$ and $I\!=\!(x_{0}^{d},x_{1}^{d},x_{2}^{d})+x_{0}^{s}(x_{0}^{2}x_{1}^{d-s-2},x_{1}^{d-s},x_{1}x_{2}^{d-s-1},x_{2}^{d-s})$.

After restricting to $x_0+x_1+x_2=0$, we see that none of them corresponds to a Togliatti system.

\vskip 2mm
To finish with the case b2, we see what happens when $d-s-2\geq e\geq\lfloor\frac{d-s-2}{2}\rfloor$. With the same argument that we use before, we can see that $a=e$. The difference with the case $u=s+1$ is that in this case
we can have $m_{1}$ and $m_{2}$ aligned vertically. This condition can be translated as the case when $d-b-s=d-a-s-2$. If this does not happen (i.e. $b>a+2$), then it will contradict the minimality
of $I$. Indeed: let us suppose that $0\leq k:=d-b-s<d-a-s-2$. Inductively we have the factorization $F_{d-s-e-1}=L_{0}^{2}\dotsb L_{k-1}^{2}F_{d-s-e-k-1}$. $F_{d-s-e-k-1}$ is a plane curve of degree $d-s-e-k-1$
which passes through all $d-s-e-k$ points of $\tilde{\tilde{A}}_{I}^{k}$. Hence, we have the factorization $F_{d-s-e-k-1}=L_{k}^{2}F_{d-s-e-k-2}$, contradicting the minimality assumption.

Therefore it must be $b=a+2$ and, since $b>c>a$ we have $c=a+1$. Finally we get: $I=
(x_{0}^{d},x_{1}^{d},x_{2}^{d})+x_{0}^{s}x_{1}^{a}x_{2}^{d-a-s-2}(x_{0}^{2},x_{1}^{2},x_{1}x_{2},x_{2}^{2})$ which is of type $(1)$.

\vskip 4mm

\noindent\textsc{Case 2C:} We assume that $u:=a_{1}=a_{2}>a_{3}=a_{4}=:s\geq 2$. Arguing as in case \textsc{2B} we get
$u=s+1$ and $I=(x_{0}^{d},x_{1}^{d},x_{2}^{d})+x_{0}^{s}(x_{0}x_{1}^{a}x_{2}^{d-a-s-1},x_{0}x_{1}^{b}x_{2}^{d-b-s-1},$ $x_{1}^{c}x_{2}^{d-c-s},x_{1}^{e}x_{2}^{d-e-s})$.  We can assume $d-s-1\geq a\geq \lfloor\frac{d-s-1}{2}\rfloor$, $a>b$ and $d-s\geq c>e\geq 0$.
  %%%%%%%%%%%%COMPLETELY MODIFIED%%%%%%%%%%%%%%
Let us suppose first that $\lfloor\frac{d-s-1}{2}\rfloor>e\geq 0$.
  We consider $$\#(F_{d-s-1}\cap\tilde{A}_{I}^{(0,1)})=
  \left\{
  \begin{array}{lll}
   d-s&e\geq1,b\geq1&c1\\
   d-s-1&e=0,b\geq1&c2\\
   d-s-1&e\geq1,b=0&c3\\
   d-s-2&e=b=0&c4
  \end{array}\right.$$

  \noindent\textsc{Case c1:} Since $F_{d-s-1}$ is a plane curve of degree $d-s-1$ which contains all $d-s$ points of $\tilde{A}_{I}^{(0,1)}$ we have the factorization $F_{d-s-1}=L_{0}^{1}F_{d-s-2}$. Now,
  intersecting $F_{d-s-2}$ with $\tilde{A}_{I}^{(1,1)}$ and using the minimality assumption, we see that the only two possibilities are either $e\geq 2$ and $b\geq 2$ or $e=b=1$. In the first case
   $F_{d-s-2}$ factorizes as $F_{d-s-2}=L_{1}^{1}F_{d-s-3}$. Repeating the same argument we get that it must be $e=b$ in any case.
  Now, we consider $\tilde{\tilde{A}}_{I}$ as before and we take $$\#(F_{d-s-e-1}\cap\tilde{\tilde{A}}_{I}^{(0,2)})=
  \left\{
  \begin{array}{lll}
   d-s-e+1&d-s-2\geq a,d-s-1\geq c\\
   d-s-e&a=d-s-1,d-s-1\geq c\\
   d-s-e&d-s-2\geq a,c=d-s\\
   d-s-e-1&a=d-s-1,c=d-s
  \end{array}\right.$$

  In the second and third cases we obtain directly a contradiction with the minimality of $I$. In the fourth case, $\tilde{\tilde{A}}_{I}^{(1,2)}$ consists in $d-s-e$ different points and we have
  $F_{d-s-e-1}=L_{1}^{2}F_{d-s-e-2}$. Since $\tilde{\tilde{A}}_{I}^{(0,2)}$ has $d-s-e-1$ different points, we get a contradiction with the minimality of $I$.
  Finally, in the first case we obtain a factorization $F_{d-s-e-1}=L_{0}^{2}F_{d-s-e-2}$ and we repeat the same argument
  until we get that $m_{1}$ and $m_{3}$ are always vertically aligned. Then, $c=a+1$ and we have the factorization $F_{d-s-e-1}=L_{0}^{2}\dotsb L_{d-s-a-2}^{2}F_{a-e}$.
  If $a\geq e+2$ we have the factorization $F_{a-e}=L_{e+1}^{1}\dotsb L_{a-1}^{1}F_{1}$, which contradicts the minimality of $I$. Therefore, $a=e+1$ and $I
  =(x_{0}^{d},x_{1}^{d},x_{2}^{d})+x_{0}^{s}(x_{0}x_{1}^{e+1}x_{2}^{d-s-e-2},x_{0}x_{1}^{e}x_{2}^{d-s-e-1},x_{1}^{e+2}x_{2}^{d-s-e-2},x_{1}^{e}x_{2}^{d-s-e})=
  (x_{0}^{d},x_{1}^{d},x_{2}^{d})+x_{0}^{s}x_{1}^{e}x_{2}^{d-s-e-2}(x_{0}x_{1},x_{0}x_{2},x_{1}^{2},x_{2}^{2})$ which is of type (1).

  \vskip 2mm
  \noindent\textsc{Case c2:} We assume $e=0$ and $b\geq 1$. Let us consider $$\#(F_{d-s-1}\cap\tilde{A}_{I}^{(1,1)})=
  \left\{
  \begin{array}{lll}
   d-s&b\geq2, c\geq2&(i)\\
   d-s-1&b=1,c\geq2&(ii)\\
   d-s-1&b\geq2,c=1&(iii)\\
   d-s-2&b=c=1&(iv)
  \end{array}\right.$$

  In Case (i) we factorize $F_{d-s-1}=L_{1}^{1}F_{d-s-2}$. Since $F_{d-s-2}$ is a plane curve of degree $d-s-2$ containing all $d-s-1$ different points of $\tilde{A}_{I}^{(0,1)}$ it factorizes as
  $F_{d-s-2}=L_{0}^{1}F_{d-s-3}$. This contradicts the minimality of $I$.

  Case (ii): assume $e=0$, $b=1$ and $c\geq 2$. We consider $\#(F_{d-s-1}\cap\tilde{A}_{I}^{(0,2)})$ and arguing as in the previous cases we get three possibilities:

  $a=d-s-1$, $c=d-s-1$ and $I=(x_{0}^{d},x_{1}^{d},x_{2}^{d})+x_{0}^{s}(x_{0}x_{1}^{d-s-1},x_{0}x_{1}x_{2}^{d-s-2}, x_{1}^{d-s-1}x_{2}, x_{2}^{d-s})$.

  $a=d-s-2$, $c=d-s$ and $I=(x_{0}^{d},x_{1}^{d},x_{2}^{d})+x_{0}^{s}(x_{0}x_{1}^{d-s-2}x_{2},x_{0}x_{1}x_{2}^{d-s-2}, x_{1}^{d-s}, x_{2}^{d-s})$.

  $a=d-s-1$, $c=d-s$ and $I=(x_{0}^{d},x_{1}^{d},x_{2}^{d})+x_{0}^{s}(x_{0}x_{1}^{d-s-1},x_{0}x_{1}x_{2}^{d-s-2}, x_{1}^{d-s}, x_{2}^{d-s})$.

  Restricting the generators to the hyperplane $x_{0}+x_{1}+x_{2}$, we see that each of them is a Togliatti system if, and only if $s=d-3$.

  Case (iii): assume $e=0$, $c=1$ and $b\geq 2$. In particular $a\geq 3$ and $s\leq d-4$. Arguing as before, we see that the only viable possibility
  is $a=d-s-1$ and $b=d-s-2$. Therefore
  $I=(x_{0}^{d},x_{1}^{d},x_{2}^{d})+x_{0}^{s}(x_{0}x_{1}^{d-s-1},x_{0}x_{1}^{d-s-3}x_{2}^{2}, x_{1}x_{2}^{d-s-1}, x_{2}^{d-s})$, which is never a Togliatti system.

  Case (iv): assume $e=0$ and $b=c=1$. Now it only remains to determinate $a$. If $d-s-2\geq a$, we have $F_{d-s-1}=L_{0}^{2}\dotsb L_{d-s-a-2}^{2}F_{a}$.
  Since $\tilde{A}_{I}^{(d-s-a-1,2)}$
  consists in $a+1$ different points, we get a contradiction with the minimality of $I$. Therefore,
  $a=d-s-1$. Using the same argumentation we see that $a=b+1=2$. Thus $s=d-4$ and
  $I=(x_{0}^{d},x_{1}^{d},x_{2}^{d})+x_{0}^{d-4}(x_{0}x_{1}^{3}, x_{0}x_{1}x_{2}^{2}, x_{1}x_{2}^{3}, x_{2}^{4})$ which is not a Togliatti system.

  \vskip 2mm
  \noindent\textsc{Case c3:}
  Now, assume that $b=0$ and $e\geq 1$. Since $a>e$ by hypothesis, considering $\#(F_{d-s-1}\cap\tilde{A}_{I}^{(1,1)})$ we see that $e=1$ and $s\leq d-3$.
  If we consider $\#(F_{d-s-1}\cap\tilde{A}_{I}^{(0,2)})$ we get that the only viable possibilities are:

  $a=d-s-1$, $c=d-s-1$ and $I=(x_{0}^{d},x_{1}^{d},x_{2}^{d})+x_{0}^{s}(x_{0}x_{1}^{d-s-1}, x_{0}x_{2}^{d-s-1}, x_{1}^{d-s-1}x_{2}, x_{1}x_{2}^{d-s-1})$.

  $a=d-s-2$, $c=d-s$ and $I=(x_{0}^{d},x_{1}^{d},x_{2}^{d})+x_{0}^{s}(x_{0}x_{1}^{d-s-2}x_{2}, x_{0}x_{2}^{d-s-1}, x_{1}^{d-s}, x_{1}x_{2}^{d-s-1})$.

  $a=d-s-1$, $c=d-s$ and $I=(x_{0}^{d},x_{1}^{d},x_{2}^{d})+x_{0}^{s}(x_{0}x_{1}^{d-s-1}, x_{0}x_{2}^{d-s-1}, x_{1}^{d-s}, x_{1}x_{2}^{d-s-1})$.

  Each of them are Togliatti systems if, and only if $s=d-3$.

  \vskip 2mm
  \noindent\textsc{Case c4:}
  in this case we assume that $e=b=0$. If $a,c\geq 3$, we have the factorization $F_{d-s-1}=L_{1}^{1}L_{2}^{1}F_{d-s-3}$. Since $F_{d-s-3}$ is a plane curve containing all $d-s-2$ points of $\tilde{A}_{I}^
  {(0,1)}$, we get $F_{d-s-3}=L_{0}^{0}F_{d-s-4}$ which contradicts the minimality of $I$. Therefore we can consider three subcases:

  Case (i). Assume that $a=1$, then it has to be either $s=d-2$ or $s=d-3$. Hence $I$ is one of the following possibilities:
  $(x_{0}^{d}, x_{1}^{d}, x_{2}^{d})+x_{0}^{d-2}(x_{0}x_{1}, x_{0}x_{2},x_{1}^{2}, x_{2}^{2})$,
  $(x_{0}^{d}, x_{1}^{d}, x_{2}^{d})+x_{0}^{d-2}x_{2}(x_{0},x_{1}, x_{2})+(x_{0}^{d-1}x_{1})$,
  $(x_{0}^{d}, x_{1}^{d}, x_{2}^{d})+x_{0}^{d-3}(x_{0}x_{1}x_{2}, x_{0}x_{2}^{2},x_{1}^{3}, x_{2}^{3})$,
  $(x_{0}^{d}, x_{1}^{d}, x_{2}^{d})+x_{0}^{d-3}x_{2}($

  \noindent$x_{0}x_{1}, x_{0}x_{2},x_{1}^{2}, x_{2}^{2})$ and
  $(x_{0}^{d}, x_{1}^{d}, x_{2}^{d})+x_{0}^{d-3}x_{2}^{2}(x_{0},x_{1}, x_{2})+(x_{0}^{d-2}x_{1}x_{2})$. And only the first and the third possibilities give to minimal Togliatti systems of type (1) and
  type (4) respectively.

  Case (ii). Assume that $a=2$, then it can be either $s=d-3$, $s=d-4$ or $s=d-5$. Therefore, $I$ is one of the next ideals:
  $(x_{0}^{d}, x_{1}^{d}, x_{2}^{d})+x_{0}^{d-3}(x_{0}x_{1}^{2}, x_{0}x_{2}^{2},x_{1}^{3}, x_{2}^{3})$,
  $(x_{0}^{d}, x_{1}^{d}, x_{2}^{d})+x_{0}^{d-3}(x_{0}x_{1}^{2}, x_{0}x_{2}^{2},x_{1}^{2}x_{2}, x_{2}^{3})$,
  $(x_{0}^{d}, x_{1}^{d}, x_{2}^{d})+x_{0}^{d-3}(x_{0},x_{1}, x_{2})+(x_{0}^{d-2}x_{1}^{2})$,
  $(x_{0}^{d}, x_{1}^{d}, x_{2}^{d})+x_{0}^{d-4}(x_{0}x_{1}^{2}x_{2},$

  \noindent$x_{0}x_{2}^{3},x_{1}^{4}, x_{2}^{4})$,
  $(x_{0}^{d}, x_{1}^{d}, x_{2}^{d})+x_{0}^{d-4}(x_{0}x_{1}^{2}x_{2}, x_{0}x_{2}^{3},x_{1}^{3}x_{2}, x_{2}^{4})$,
  $(x_{0}^{d}, x_{1}^{d}, x_{2}^{d})+x_{0}^{d-4}(x_{0}x_{1}^{2}x_{2}, x_{0}x_{2}^{3},x_{1}^{2}x_{2}^{2},$

  \noindent$x_{2}^{4})$,
  $(x_{0}^{d}, x_{1}^{d}, x_{2}^{d})+x_{0}^{d-4}x_{2}^{3}(x_{0},x_{1}, x_{2})+(x_{0}^{d-3}x_{1}^{2}x_{2})$,
  $(x_{0}^{d}, x_{1}^{d}, x_{2}^{d})+x_{0}^{d-5}(x_{0}x_{1}^{2}x_{2}^{2}, x_{0}x_{2}^{4},x_{1}^{5}, x_{2}^{5})$,
  $(x_{0}^{d}, x_{1}^{d}, x_{2}^{d})+x_{0}^{d-5}(x_{0}x_{1}^{2}x_{2}^{2}, x_{0}x_{2}^{4},x_{1}^{4}x_{2}, x_{2}^{5})$,
  $(x_{0}^{d}, x_{1}^{d}, x_{2}^{d})+x_{0}^{d-5}(x_{0}x_{1}^{2}x_{2}^{2}, x_{0}x_{2}^{4},x_{1}^{3}x_{2}^{2}, x_{2}^{5})$,
  $(x_{0}^{d}, x_{1}^{d}, $

  \noindent$x_{2}^{d})+x_{0}^{d-5}(x_{0}x_{1}^{2}x_{2}^{2}, x_{0}x_{2}^{4},x_{1}^{2}x_{2}^{3}, x_{2}^{5})$ and
  $(x_{0}^{d}, x_{1}^{d}, x_{2}^{d})+x_{0}^{d-5}x_{2}^{4}(x_{0},x_{1}, x_{2})+(x_{0}^{d-4}x_{1}^{2}x_{2}^{2})$.

 Only the first and the second ones correspond to minimal Togliatii systems.

  Case (iii). Assume that $a\geq 3$ which implies that either $c=1$ or $c=2$ and we have $s\leq d-4$. In both cases, since
    $a\geq 3$ and $c\leq 2$, $m_{1}$ cannot be aligned vertically with any $m_{i}$. Therefore, in both cases, we get a contradiction with the minimality
  of $I$.

  \vskip 4mm
  To finish case \textsc{2C}, let us assume $e\geq\lfloor\frac{d-s-1}{2}\rfloor>0$. We will separate two cases: when $b=0$ and when $b\geq 1$.

  Case (i). We assume $b=0$, then considering $\#(F_{d-s-1}\cap\tilde{A}_{I}^{(1,1)})$ and using the bound for $e$, we obtain that $a=1$ and therefore it is either
  $s=d-2$ or $d-3$. So, $I$ is one of the following ideals: $(x_{0}^{d}, x_{1}^{d}, x_{2}^{d})+x_{0}^{d-2}x_{1}(x_{0}, x_{1}, x_{2})+(x_{0}^{d-1}x_{2})$,
  $(x_{0}^{d}, x_{1}^{d}, x_{2}^{d})+x_{0}^{d-3}(x_{0}x_{1}x_{2}, x_{0}x_{2}^{2}, x_{1}^{3}, x_{1}^{2}x_{2})$,
  $(x_{0}^{d}, x_{1}^{d}, x_{2}^{d})+x_{0}^{d-3}(x_{0}x_{1}x_{2}, x_{0}x_{2}^{2}, x_{1}^{3}, x_{1}x_{2}^{2})$ and
  $(x_{0}^{d}, x_{1}^{d}, x_{2}^{d})+x_{0}^{d-3}x_{1}$

  \noindent$x_{2}(x_{0}, x_{1}, x_{2})+(x_{0}^{d-2}x_{2}^{2})$. And any of them are minimal Togliatti systems.

  Case (ii). We assume $b\geq 1$. In this case, we can assume $e\geq b$ (the other case is symmetric) and we obtain the factorization $F_{d-s-1}=L_{0}^{1}\dotsb L_{b-1}^{1}F_{d-s-b-1}$. If $e>b$ we get a
  contradiction with the minimality of $I$. Hence, $e=b$. Now we consider as before $\tilde{\tilde{A}}_{I}$ and we have
  $$\#(F_{d-s-b-1}\cap\tilde{\tilde{A}}_{I}^{(0,2)})=
  \left\{
  \begin{array}{lll}
   d-s-b+1&d-s-2\geq a,d-s-1\geq c\\
   d-s-b&a=d-s-1,d-s-1\geq c\\
   d-s-b&d-s-2\geq a,c=d-s\\
   d-s-b-1&a=d-s-1,c=d-s
  \end{array}\right.$$
  In the second and third cases we get immediately a contradiction with the minimality. In the first case, we can repeat the same argument and we get contradiction unless $m_{1}$ and $m_{3}$ are aligned
  vertically. Hence, we always obtain that $c=a+1$, and we have the factorization $F_{d-s-b-1}=L_{0}^{2}\dotsb L_{d-s-a-2}^{2}F_{a-b}$. If $a\geq b+2$, then $\tilde{\tilde{A}}_{I}^{(d-s-a,2)}$ consists in
  $a-b+1$ different points, so we have the factorization $F_{a-b}=L_{d-s-a}^{2}F_{a-b-1}$. Now $F_{a-b-1}$ is a plane curve of degree $a-b-1$ which contains all $a-b$ points of $\tilde{\tilde{A}}_{I}^{
  (d-s-a-1,2)}$ and then it factorizes as $F_{a-b-1}=L_{d-s-a-1}^{2}F_{a-b-2}$ which contradicts the minimality of $I$. Therefore, $a=b+1$ and $I=
  (x_{0}^{d}, x_{1}^{d}, x_{2}^{d})+x_{0}^{s}(x_{0}x_{1}^{b+1}x_{2}^{d-s-b-2}, x_{0}x_{1}^{b}x_{2}^{d-s-b-1}, x_{1}^{b+2}x_{2}^{d-s-b-2}, x_{1}^{b}x_{2}^{d-s-b})=
  (x_{0}^{d}, x_{1}^{d}, x_{2}^{d})+x_{0}^{s}x_{1}^{b}x_{2}^{d-s-b-2}(x_{0}x_{1}, x_{0}x_{2}, x_{1}^{2}, x_{2}^{2})$ which is of type (1).

  %%%%%%%%%%%%%%%%%%%%%%%%%%%%%%%%%%%%%%%%%%%%%

\vskip 4mm

\noindent\textsc{Case 2D:} We assume that $u:=a_{1}>v:=a_{2}>a_{3}=a_{4}=:s\geq 2$. Recall that we have the factorization $F_{d-1}=L_{0}^{0}L_{1}^{0}\dotsb L_{s-1}^{0}F_{d-s-1}$ and we easily check that the minimality of $I$ forces  $v=s+1$. So we can write
$I\!=\!(x_{0}^{d},x_{1}^{d},x_{2}^{d})+x_{0}^{s}(x_{0}^{r}x_{1}^{a}x_{2}^{d-a-s-r},x_{0}x_{1}^{b}x_{2}^{d-b-s-1},x_{1}^{c}x_{2}^{d-c-s},x_{1}^{e}x_{2}^{d-e-s})$ with $u=s+r$ and $d-s-1\geq r\geq 2$. We can  assume
$d-s-1\geq b\geq\lfloor\frac{d-s-1}{2}\rfloor$ and $d-s\geq c > e\geq 0$, and we have  $d-s-r\geq a\geq 0$ and $s\leq d-3$.
Let us suppose first that $\lfloor\frac{d-s-1}{2}\rfloor>e\geq0$. We consider $$\#(F_{d-s-1}\cap \tilde{A}_{I}^{(0,1)})=
\left\{
\begin{array}{lll}
 d-s&e\geq 1, a\geq 1&(d1)\\
 d-s-1&e=0, a\geq1&(d2)\\
 d-s-1&e\geq1, a=0&(d3)\\
 d-s-2&e=a=0&(d4).
\end{array}\right.$$
\noindent\textsc{Case} d1: In this case  $a=e$. Indeed, if  $a>e\geq1$ (the other case is symmetric)
we have the factorization $F_{d-s-1}=L_{0}^{1}\dotsc L_{e-1}^{1}F_{d-s-e-1}=L_{0}^{1}\dotsc L_{e-1}^{1}   L_{e}^{1}F_{d-s-e-2}$ which contradicts the minimality of $I$. Let us now consider $$\#(F_{d-s-e-1}\cap\A_{I}^{(e+1,1)})=
\left\{
\begin{array}{lll}
 d-s-e&b\geq e+2,c\geq e+2&(i)\\
 d-s-e-1& b=e+1,c\geq e+2&(ii)\\
 d-s-e-1& b\geq e+2,c=e+1&(iii)\\
 d-s-e-2&b=c=e+1&(iv).
\end{array}
\right.$$

Case (i). We have  $F_{d-s-e-1}=L_{e+1}^{1}F_{d-s-e-2}$, and since $F_{d-s-e-2}$ passes through all $d-s-e-1$ points of $\A_{I}^{(e,1)}$ we contradicts the minimality of $I$.

Case (ii). We assume $b=e+1$. Let us consider $\tilde{\A}_{I}$ as we did before and we examine $$\#(F_{d-s-e-1}\cap\tilde{\A}_{I}^{(0,2)})=
\left\{
\begin{array}{ll}
 d-s-e+1&d-s-1\geq c, 1\leq d-s-e-r\\
 d-s-e&c=d-s,1\leq d-s-e-r\\
 d-s-e&d-s-1\geq c, d-s-e-r=0\\
 d-s-e-1& c=d-s, d-s-e-r=0
\end{array}
\right.$$

In the second and third possibilities we obtain directly a contradiction with the minimality of $I$. In the last possibility we also obtain a contradiction. In fact, if $c=d-s$ and $s+r=d-e$, we do not remove
any point of $H_{1}^{2}$ and we have $\#(F_{d-s-e-1}\cap\tilde{\A}_{I}^{(1,2)})=d-s-e$. Then $F_{d-s-e-1}=L_{1}^{2}F_{d-s-e-2}=L_1^2L_{0}^{2}F_{d-s-e-3}$, which contradicts the minimality
of $I$.

Therefore if $b=e+1$, it must be $d-s-1\geq c\geq e+2$ and $1\leq d-s-e-r$. Iterating this argument we  conclude that either $c=e+2$ and $r=2$ or $c=e+3$ and $r=3$. Therefore, either
$I=(x_{0}^{d},x_{1}^{d},x_{2}^{d})+x_{0}^{s}x_{1}^{e}x_{2}^{d-s-e-2}(x_{0}^{2},x_{0}x_{1},x_{1}^{2},x_{2}^{2})$ which is of type $(2)$; or
$I=(x_{0}^{d},x_{1}^{d},x_{2}^{d})+x_{0}^{s}x_{1}^{e}x_{2}^{d-s-e-3}(x_{0}^{3},x_{0}x_{1}x_{2},x_{1}^{3},x_{2}^{3})$ which is of type $(3)$.

Case (iii). Arguing as in case (ii) we get  $b=e+2$ and $r=2$. Therefore, $I=(x_{0}^{d},x_{1}^{d},x_{2}^{d})+
x_{0}^{s}x_{1}^{e}x_{2}^{d-s-e-2}(x_{0}^{2},x_{1}x,x_{1}^{2},x_{2}^{2})$ which is of type $(2)$.

Case (iv). Arguing as in case (ii) we get that $r=2$ and $I$ is of type $(1)$.

\vskip 2mm
\noindent\textsc{Case} d2: In this case we assume $e=0$ and $a\geq 1$. We will separate the case $b=1$ from the case $b\geq2$.

If $b=1\geq\lfloor\frac{d-s-1}{2}\rfloor$ then $s=d-3$ and $r=2$. Hence, $I=(x_{0}^{d},x_{1}^{d},x_{2}^{d})+x_{0}^{d-3}(x_{0}x_{1}x_{2},x_{1}^{3},x_{2}^{3})+
(\!x_{0}^{d-3}x_{0}^{2}x_{1})$,$(\!x_{0}^{d},x_{1}^{d},x_{2}^{d})\!+x_{0}^{d-3}(\!x_{0}^{2}x_{1},x_{0}x_{1}x_{2},x_{1}^{2}x_{2},x_{2}^{3})$ or $(\!x_{0}^{d},x_{1}^{d},x_{2}^{d})\!+x_{0}^{d-3}
(\!x_{0}^{2}x_{1},x_{0}x_{1}x_{2},x_{1}x_{2}^{2},x_{2}^{3})$. 
The first one is not minimal and the remaining two are of type $(4)$.

Assume $b\geq 2$. Let us first suppose $d-s-r-1\geq0$ (i.e. $m_{1}\notin H_{0}^{2}$) and we consider
$$\#(F_{d-s-1}\cap\A_{I}^{(1,1)})=\left\{
\begin{array}{lll}
 d-s&a\geq 2,c\geq 2&(i)\\
 d-s-1&a=1,c\geq2&(ii)\\
 d-s-1&a\geq2,c=1&(iii)\\
 d-s-2&a=c=1&(iv).
\end{array}\right.$$

Case (i). We get $F_{d-s-1}=L_{1}^{1}F_{d-s-2}=L_{0}^{1}L_{1}^{1}F_{d-s-3}$ which contradicts the minimality of $I$.

Case (ii). Assume that $a=1$ and $c\geq 2$. Suppose that $d-s-r-1>0$ and let us consider
$$\#(F_{d-s-1}\cap\A_{I}^{(0,2)})=
\left\{
\begin{array}{ll}
 d-s&d-s-2\geq b, d-s-1\geq c\\
 d-s-1&b=d-s-1,d-s-1\geq c\\
 d-s-1&d-s-2\geq b, c=d-s\\
 d-s-2&b=d-s-1, c=d-s.
\end{array}\right.$$

The first possibility contradicts  the minimality of $I$.

Now let us suppose that $d-s-r-1>1$. In this case, the second (resp.  third) possibility can occur if, and only if $b=c=d-s-1$ (resp. $b=d-s-2$ and $c=d-s$). Therefore $I$ is one of the next types:

$I=(x_{0}^{d},x_{1}^{d},x_{2}^{d})+x_{0}^{s}(x_{0}^{r}x_{1}x_{2}^{d-s-r-1},x_{0}x_{1}^{d-s-1},x_{1}^{d-s-1}x_{2},x_{2}^{d-s})$ which does not correspond to a Togliatti system.

$I=(x_{0}^{d},x_{1}^{d},x_{2}^{d})+x_{0}^{s}(x_{0}^{r}x_{1}x_{2}^{d-s-r-1},x_{0}x_{1}^{d-s-2}x_{2},x_{1}^{d-s},x_{2}^{d-s})$ which is a Togliatti system if, and only if $r=2$ and $s=d-5$ (of type $(6)$).

If $d-s-r-1=1$, then there are no special restrictions for the second and third case. Therefore $I$ is one of the next types:

$I=(x_{0}^{d},x_{1}^{d},x_{2}^{d})+x_{0}^{s}(x_{0}^{d-s-2}x_{1}x_{2},x_{0}x_{1}^{d-s-1},x_{1}^{c}x_{2}^{d-s-c},x_{2}^{d-s})$ which is a Togliatti system if, and only if $s=d-4$ and $c=3$ (of type $(5)$), or

$I=(x_{0}^{d},x_{1}^{d},x_{2}^{d})+x_{0}^{s}(x_{0}^{d-s-2}x_{1}x_{2},x_{0}x_{1}^{b}x_{2}^{d-s-b-1},x_{1}^{d-s},x_{2}^{d-s})$ which is a Togliatti system if, and only if $s=d-5$ and $b=2$ (of type $(6)$), or

$I=(x_{0}^{d},x_{1}^{d},x_{2}^{d})+x_{0}^{s}(x_{0}^{r}x_{1}x_{2}^{d-s-r-1},x_{0}x_{1}^{d-s-1},x_{1}^{d-s},x_{2}^{d-s})$ which is a Togliatti system if, and only if $r=2$ and $s=d-3$ (of type $(3)$).

Now, let us suppose that $d-s-r-1=0$. Arguing as usual, we see that it cannot be $d-s-2\geq c$ and $d-s-3\geq b$. Therefore $d-s\geq c\geq d-s-1$ or $d-s-1\geq b\geq d-s-2$, and we have the following possibilities:

 \noindent$b=d-s-2$, $d-s-2\geq c$ and $I=(x_{0}^{d},x_{1}^{d},x_{2}^{d})+x_{0}^{s}(x_{0}^{d-s-1}x_{1},x_{0}x_{1}^{d-s-2}x_{2},x_{1}^{c}x_{2}^{d-s-c},x_{2}^{d-s})$ which is a Togliatti system if, and only if $s=d-3$ and $c=d-s-2$ (of type $(4)$).

 \noindent$d-s-3\geq b$, $c=d-s-1$ and $I=(x_{0}^{d},x_{1}^{d},x_{2}^{d})+x_{0}^{s}(x_{0}^{d-s-1}x_{1},x_{0}x_{1}^{b}x_{2}^{d-s-b-1},x_{1}^{d-s-1}x_{2},x_{2}^{d-s})$ which is a Togliatti system if, and only if $d-3\geq s\geq d-4$ and $b=d-s-3$ (resp.
 of type $(4)$ and $(5)$).

 \noindent$b=d-s-2$, $c=d-s-1$ and $I=(x_{0}^{d},x_{1}^{d},x_{2}^{d})+x_{0}^{s}(x_{0}^{d-s-1}x_{1},x_{0}x_{1}^{d-s-2}x_{2},x_{1}^{d-s-1}x_{2},x_{2}^{d-s})$ which is a Togliatti system if, and only if $s=d-3$ (of type $(4)$).

 \noindent$b=d-s-1$, $d-s-1\geq c$ and $I=(x_{0}^{d},x_{1}^{d},x_{2}^{d})+x_{0}^{s}(x_{0}^{d-s-1}x_{1},x_{0}x_{1}^{d-s-1},x_{1}^{c}x_{2}^{d-s-c},x_{2}^{d-s})$ which is a Togliatti system if, and only if $s=d-3$ and $d-s-1\geq c\geq d-s-2$ (of type $(4)$).

 \noindent$d-s-2\geq b$, $c=d-s$ and $I=(x_{0}^{d},x_{1}^{d},x_{2}^{d})+x_{0}^{s}(x_{0}^{d-s-1}x_{1},x_{0}x_{1}^{b}x_{2}^{d-s-b-1},x_{1}^{d-s},x_{2}^{d-s})$ which is a Togliatti system if, and only if $s=d-3$ and $d-s-2\geq b\geq d-s-3$ (of type $(4)$).

 \noindent$b=d-s-1$, $c=d-s$ and $I=(x_{0}^{d},x_{1}^{d},x_{2}^{d})+x_{0}^{s}(x_{0}^{d-s-1}x_{1},x_{0}x_{1}^{d-s-1},x_{1}^{d-s},x_{2}^{d-s})$ which is a Togliatti system if, and only if $s=d-3$ (of type $(4)$).

Case (iii). Assume $c=1$ and $a\geq2$. We consider  $\#(F_{d-s-1}\cap\A_{I}^{(0,2)})$ and  we obtain that $I$ is one of the next types:

$I=(x_{0}^{d},x_{1}^{d},x_{2}^{d})+x_{0}^{s}(x_{0}^{r}x_{1}^{d-s-r-1}x_{2},x_{0}x_{1}^{d-s-1},x_{1}x_{2}^{d-s-1},x_{2}^{d-s})$ which is a Togliatti system if, and only if $r=2$ and $d-3\geq s\geq d-4$ (of type $(4)$ and $(5)$).

$I=(x_{0}^{d},x_{1}^{d},x_{2}^{d})+x_{0}^{s}(x_{0}^{r}x_{1}^{d-s-r},x_{0}x_{1}^{d-s-2}x_{2},x_{1}x_{2}^{d-s-1},x_{2}^{d-s})$ which is a Togliatti system if, and only if $r=2$ and $s=d-3$ (of type $(4)$).

$I=(x_{0}^{d},x_{1}^{d},x_{2}^{d})+x_{0}^{s}(x_{0}^{r}x_{1}^{d-s-r},x_{0}x_{1}^{d-s-1},x_{1}x_{2}^{d-s-1},x_{2}^{d-s})$. In this case, let us consider $\#(F_{d-s-1}\cap\A_{I}^{(1,2)})=d-s$ and inductively we obtain
$F_{d-s-1}=L_{1}^{2}\dotsb L_{d-s-2}^{1}F_{1}$. Therefore it must be $r=2$ and $s=d-3$, and $I$ is of type $(4)$.

Case (iv). Assume $a=c=1$. Let us first suppose that $d-s-r-1>0$.
If $d-s-2\geq b$ we  factorize $F_{d-s-1}$ as $F_{d-s-1}=L_{0}^{2}F_{d-s-2}$ which  contradicts the minimality of $I$. Therefore, $b=d-s-1$ and
we  factorize $F_{d-s-1}=L_{1}^{2}\dotsb L_{d-s-r-2}F_{r+1}$. Since $\#(F_{r}\cap\A_{I}^{(0,2)})=d-s-1$ we have $r+1\geq d-s-1$ and then $d-s-r-1\leq 1$. Therefore, $d-s-r-1=1$ and $I=
(x_{0}^{d},x_{1}^{d},x_{2}^{d})+x_{0}^{s}(x_{0}^{d-s-2}x_{1}x_{2},x_{0}x_{1}^{d-s-1},x_{1}x_{2}^{d-s-1},x_{2}^{d-s})$. It is a Togliatti system if, and only if $s= d-4$ (of type $(5)$).

If $d-s-r-1=0$,  we use the same argumentation to prove that $d-s-1\geq b\geq d-s-2$ and then we have two possibilities:

\noindent$b=d-s-1$ and $I=(x_{0}^{d},x_{1}^{d},x_{2}^{d})+x_{0}^{s}(x_{0}^{d-s-1}x_{1},x_{0}x_{1}^{d-s-1},x_{1}x_{2}^{d-s-1},x_{2}^{d-s})$

\noindent$b=d-s-2$ and $I=(x_{0}^{d},x_{1}^{d},x_{2}^{d})+x_{0}^{s}(x_{0}^{d-s-1}x_{1},x_{0}x_{1}^{d-s-2}x_{2},x_{1}x_{2}^{d-s-1},x_{2}^{d-s})$

They are Togliatti systems if, and only if $s=d-3$ (of type $(4)$).

\vskip 2mm
\noindent \textsc{Case} d3: Let us assume $e\geq 1$ and $a=0$. Actually, it must be $e=1$. Otherwise, $e>1$ and $\#(F_{d-s-1}\cap\A_{I}^{(1,1)})=d-s$, and we have seen that this cannot happen.

Now, let us suppose $d-s-r>1$. Arguing as before we  see that there are three possibilities:

\noindent$b=d-s-1$, $c=d-s-1$ and $I=(x_{0}^{d},x_{1}^{d},x_{2}^{d})+x_{0}^{s}(x_{0}^{r}x_{2}^{d-s-r}, x_{0}x_{1}^{d-s-1},$ $x_{1}^{d-s-1}x_{2},x_{1}x_{2}^{d-s-1})$.

\noindent$b=d-s-2$, $c=d-s$ and $I=(x_{0}^{d},x_{1}^{d},x_{2}^{d})+x_{0}^{s}(x_{0}^{r}x_{2}^{d-s-r}, x_{0}x_{1}^{d-s-2}x_{2},$ $x_{1}^{d-s},x_{1}x_{2}^{d-s-1})$.

They do not correspond to a Togliatti system.

\noindent$b=d-s-1$, $c=d-s$. If $d-2>s+r$, then we have the factorization $F_{d-s-1}=L_{1}^{2}\dotsb L_{d-s-r-1}F_{r}$ and $\#(F_{r}\cap \A_{I}^{(0,2)})=d-s-2>r$, which contradicts the minimality of $I$.
Hence we have $s+r=d-2$ and $I=(x_{0}^{d},x_{1}^{d},x_{2}^{d})+x_{0}^{s}(x_{0}^{d-s-2}x_{2}^{2}, x_{0}x_{1}^{d-s-1}, x_{1}^{d-s}, x_{1}x_{2}^{d-s-1})$ which is never 
a Togliatti system since $s\leq d-4$.

To finish, assume $d-s-r=1$. Arguing in the same manner, we see that it cannot occur $d-s-3\geq b$ and $d-s-2\geq c\geq 1$. Therefore we have the following possibilities:

\noindent$d-s-3\geq b$, $c=d-s-1$ and $I\!=\!(x_{0}^{d},x_{1}^{d},x_{2}^{d})+x_{0}^{s}(x_{0}^{d-s-1}x_{2}, x_{0}x_{1}^{b}x_{2}^{d-s-b-1},x_{1}^{d-s-1}x_{2},x_{1}x_{2}^{d-s-1})$, it is a Togliatti system if, and only if $s=d-3$ and $b=d-s-3$ (of type $(4)$).

\noindent$b=d-s-2$, $c=d-s-1$ and $I=(x_{0}^{d},x_{1}^{d},x_{2}^{d})+x_{0}^{s}(x_{0}^{d-s-1}x_{2}, x_{0}x_{1}^{d-s-2}x_{2},x_{1}^{d-s-1}x_{2},x_{1}x_{2}^{d-s-1})$, it is a Togliatti system 
if, and only if $s=d-3$ but it is not minimal.

\noindent$b=d-s-1$, $d-s-1\geq c$ and $I=(x_{0}^{d},x_{1}^{d},x_{2}^{d})+x_{0}^{s}(x_{0}^{d-s-1}x_{2}, x_{0}x_{1}^{d-s-1},x_{1}^{c}x_{2}^{d-s-c},x_{1}x_{2}^{d-s-1})$, it is a Togliatti system if, and only if $s=d-3$ and $c=d-s-1$ (of type $(4)$).

\noindent$d-s-2\geq b$, $c=d-s$ and $I=(x_{0}^{d},x_{1}^{d},x_{2}^{d})+x_{0}^{s}(x_{0}^{d-s-1}x_{2}, x_{0}x_{1}^{b}x_{2}^{d-s-b-1},x_{1}^{d-s},x_{1}x_{2}^{d-s-1})$. It is a Togliatti system if, and only if $s=d-3$ and $d-s-2\geq b\geq d-s-3$ (of type $(4)$), or
$s=d-4$ and $b=d-s-2$ (of type $(5)$).

\noindent$b=d-s-1$, $c=d-s$ and $I=(x_{0}^{d},x_{1}^{d},x_{2}^{d})+x_{0}^{s}(x_{0}^{d-s-1}x_{2}, x_{0}x_{1}^{d-s-2}x_{2},x_{1}^{d-s},x_{1}x_{2}^{d-s-1})$, it is a Togliatti system if, and only if $s=d-3$ (of type $(4)$).

\vskip 2mm
\noindent\textsc{Case} d4: Let us assume $e=a=0$. If $b\geq 3$ and $c\geq 3$,  we have the factorization $F_{d-s-1}=L_{1}^{1}L_{2}^{1}F_{d-s-3}$ and $\#(F_{d-s-3}\cap\A_{I}^{(0,1)})=d-s-2$ and we
contradict the minimality of $I$. Hence we distinguish three cases: $b=1$, $b=2$ and $b\geq 3$.

Case (i). We assume $b=1$. Since $b\geq\lfloor\frac{d-s-1}{2}\rfloor$ it must be $s=d-3$. Therefore $I=x_{0}^{d},x_{1}^{d},x_{2}^{d})+x_{0}^{d-3}(x_{0}x_{1}x_{2},x_{1}^{3},x_{2}^{3})+(x_{0}^{d-1}x_{2})$,
$(x_{0}^{d},x_{1}^{d},x_{2}^{d})+x_{0}^{d-3}(x_{0}^{2}x_{2},x_{0}x_{1}x_{2},x_{1}^{2}x_{2},$ $x_{2}^{3})$ or $(x_{0}^{d},x_{1}^{d},x_{2}^{d})+x_{0}^{d-3}(x_{0}^{2}x_{2},x_{0}x_{1}x_{2},x_{1}x_{2}^{2},x_{2}^{3})$. The first one is not minimal while the remaining two are of type $(1)$.

Case (ii). We assume $b=2$. Since $b\geq\lfloor\frac{d-s-1}{2}\rfloor$ it must be $d-3\geq s\geq d-5$.

If $s=d-3$, $I=(x_{0}^{d},x_{1}^{d},x_{2}^{d})+x_{0}^{d-3}(x_{0}^{2}x_{2},x_{0}x_{1}^{2},x_{1}^{3},x_{2}^{3})$, $(x_{0}^{d},x_{1}^{d},x_{2}^{d})+x_{0}^{d-3}(x_{0}^{2}x_{2},x_{0}x_{1}^{2},x_{1}^{2}x_{2},x_{2}^{3})$ or
$(x_{0}^{d},x_{1}^{d},x_{2}^{d})+x_{0}^{d-3}(x_{0}^{2}x_{2},x_{0}x_{1}^{2},x_{1}x_{2}^{2},x_{2}^{3})$. All of them are of type $(4)$.

If $s=d-4$, $I=(x_{0}^{d},x_{1}^{d},x_{2}^{d})+x_{0}^{d-4}(x_{0}^{3}x_{2},x_{0}x_{1}^{2}x_{2},x_{1}^{4},x_{2}^{4})$,
$(x_{0}^{d},x_{1}^{d},x_{2}^{d})+x_{0}^{d-4}(x_{0}^{3}x_{2},x_{0}x_{1}^{2}x_{2},$

\noindent$x_{1}^{3}x_{2},x_{2}^{4})$,
$(x_{0}^{d},x_{1}^{d},x_{2}^{d})+x_{0}^{d-4}(x_{0}^{3}x_{2},x_{0}x_{1}^{2}x_{2},x_{1}^{2}x_{2}^{2},x_{2}^{4})$,
$(x_{0}^{d},x_{1}^{d},x_{2}^{d})+x_{0}^{d-4}(x_{0}^{3}x_{2},x_{0}x_{1}^{2}x_{2},$ $x_{1}x_{2}^{3},x_{2}^{4})$,
$(x_{0}^{d},x_{1}^{d},x_{2}^{d})+x_{0}^{d-4}(x_{0}^{2}x_{2}^{2},x_{0}x_{1}^{2}x_{2},x_{1}^{4},x_{2}^{4})$,
$(x_{0}^{d},x_{1}^{d},x_{2}^{d})+x_{0}^{d-4}(x_{0}^{2}x_{2}^{2},x_{0}x_{1}^{2}x_{2},x_{1}^{3}x_{2},x_{2}^{4})$,
$(x_{0}^{d},x_{1}^{d},$ $x_{2}^{d})+x_{0}^{d-4}(x_{0}^{2}x_{2}^{2},x_{0}x_{1}^{2}x_{2},x_{1}^{2}x_{2}^{2},x_{2}^{4})$,
$(x_{0}^{d},x_{1}^{d},x_{2}^{d})+x_{0}^{d-4}(x_{0}^{2}x_{2}^{2},x_{0}x_{1}^{2}x_{2},x_{1}x_{2}^{3},x_{2}^{4})$. Only the fifth one is a minimal Togliatti system, and it is of type $(5)$.

Finally, if $s=d-5$, $I$ has $15$ possibilities, but any of them is a minimal Togliatti system.

Case (iii). We assume $b\geq 3$. Then, either $c=1$ or $c=2$.

Case $c=1$. We will see that  $b=d-s-1$. Suppose  $d-s-2\geq b\geq 3$, then $\#(F_{d-s-1}\cap\A_{I}^{(0,2)})=d-s$ and $F_{d-s-1}=L_{0}^{2}F_{d-s-2}$. First we will see that
this implies that $m_{1}$ and $m_{2}$ are aligned vertically (i.e. $d-s-b-1=d-s-r$). We  suppose that $d-s-b-1\leq d-s-r$, and then $b\geq r-1$ (the other case is symmetric). Inductively we
obtain $F_{d-s-1}=L_{0}^{2}L_{1}^{2}\dotsb L_{d-s-b-2}^{2}F_{b}$. If $b>r-1$, then $\#\A_{I}^{(d-s-b-1,2)}=b+1$ and it would mean to a contradiction with the minimality of $I$. Hence,
$b=r-1$ and we get the factorization $F_{d-s-1}=L_{0}^{2}L_{1}^{2}\dotsb L_{d-s-b-2}^{2}L_{d-s-b}^{2}\dotsb L_{d-s-2}F_{1}$. Since $\#\A_{I}^{(d-s-b-1,2)}=b\geq 3$ we have again a contradiction with the minimality.

Once we have seen that $b=d-s-1$, using the usual argumentation we see that $d-s-r=1$. Therefore $I=(x_{0}^{d},x_{1}^{d},x_{2}^{d})+x_{0}^{s}(x_{0}^{d-s-1}x_{2},x_{0}x_{1}^{d-s-1},x_{1}x_{2}^{d-s-1},x_{2}^{d-s})$ with $s\leq d-3$,
which is Togliatti if, and only if $s=d-3$ and it is of type $(4)$.

Case $c=2$. Since $\#\A_{I}^{(1,1)}=d-s$ we have  $F_{d-s-1}=L_{1}^{1}F_{d-s-2}$. If $d-s-2\geq b\geq 3$, then $\#\tilde{\A}_{I}^{(0,2)}=d-s-1$ and $F_{d-s-1}$ would factorize as
$F_{d-s-1}=L_{1}^{1}L_{0}^{2}F_{d-s-3}$. This contradicts the minimality of $I$ because $\#\tilde{\A}_{I}^{(0,1)}=d-s-2$ which would force the factorization $F_{d-s-1}=L_{1}^{1}L_{0}^{2}L_{0}^{1}F_{d-s-4}$.
Therefore $b=d-s-1$ and again by minimality  we see that $d-s-r=1$. Hence, $I=(x_{0}^{d},x_{1}^{d},x_{2}^{d})+x_{0}^{s}(x_{0}^{d-s-1}x_{2}, x_{0}x_{1}^{d-s-1},x_{1}^{2}x_{2}^{d-s-2},x_{2}^{d-s})$, which is Togliatti if, and only if $s=d-3$ and in this case
it is of type $(4)$

\vskip 2mm
To finish case \textsc{2D} we see what happens when $d-s\geq c>e\geq\lfloor\frac{d-s-1}{2}\rfloor$. We  see using the minimality that either $a\geq b=e$, $b\geq a=e$ or $e\geq a=b$.

Arguing as before we  see that in the first possibility $m_{1}$ and $m_{3}$ must be vertically aligned and in particular $c=e+2$, $a=e$ and $r=2$. Therefore $I=(x_{0}^{d},x_{1}^{d},x_{2}^{d})+x_{0}^{s}x_{1}^{e}x_{2}^{d-s-e-2}(x_{0}^{2},x_{0}x_{1},x_{1}^{2},x_{2}^{2})$ which is of type $(1)$.

Now we assume $b\geq a=e$. If $b,c\geq e+1$, then we have the factorization $F_{d-s-e-1}=L_{e+1}^{1}F_{d-s-e-2}$ and, since $\#\tilde{\A}_{I}^{(e,1)}=d-s-e-1$ we get $F_{d-s-e-1}=
L_{e}^{1}L_{e+1}^{1}F_{d-s-e-3}$ which contradicts the minimality. Now, suppose $b=e+1$ and $c\geq e+2$ (resp. $b\geq e+2$ and $c=e+1$).
As we have seen earlier, $m_{1}$ and $m_{3}$ (resp. $m_{2}$) must be aligned. Therefore, we can see using the minimality assumption that $r=2$ and $c=e+2$ (resp. $r=2$ and $b=e+2$). In both cases $I$ is of type $(1)$.

Finally, let us assume $e\geq a=b$. If $e\geq a+2$, we  get a contradiction with the minimality of $I$. Hence either $e=a$ or $e=a+1$. If $e=a$ we see that $c=a+2$ and $r=2$. Therefore
$I$ is of type $(1)$. Otherwise $e=a+1$ and we get $c=a+2$ and $r=2$ and $I$ is of type $(1)$.
\end{proof}

For any integer $d\ge 3$, set $M^0(d)=\{ x_0^ax_1^bx_2^c \mid a+b+c=d \text{ and } a,b,c\ge 1\}$.

\begin{thm}\label{mainthm2}
Let $I\subset k[x_{0},\dotsc,x_{n}]$ be a smooth minimal  monomial
Togliatti system of forms of degree $d\ge 10$.
 Assume that $\mu(I)=2n+3$. Then, $n=2$ and, up to
a permutation of the coordinates, one of the following cases holds:
 \begin{itemize}
 \item[(i)] $I=(x_0^d,x_1^d, x_2^d)+m(x_0^2,x_1^2,x_0x_2,x_1x_2)$ with $m\in M^0(d-2)$, or
 \item[(ii)] $I=(x_0^d,x_1^d,x_2^d)+m(x_0^2,x_1^2,x_0x_1,x_2^2)$ with $m\in M^0(d-2)$, or
 \item[(iii)]  $I=(x_0^d,x_1^d, x_2^d)+m(x_0^3,x_1^3,x_2^3,x_0x_1x_2)$ with $m\in M^0(d-3)$.
 \end{itemize}
\end{thm}

\begin{proof} By \cite[Proposition 4.1]{MeMR}, for $n\ge 3$ and $d\ge 4$ there are no {\em smooth}  minimal  monomial
Togliatti systems $I\subset k[x_{0},\dotsc,x_{n}]$ of forms of degree $d$ with  $\mu(I)=2n+3$. So, $n=2$.
For $n=2$, the result follows from  Theorem \ref{mainthm1} together with  the smoothness criterion Proposition \ref{smoothness}.
\end{proof}

The following remarks shows that in the above Theorem the hypothesis of being smooth cannot be deleted.

\begin{rem} \label{rem1} \rm If $n=3$ and $d\ge 10$ one can easily check  that
$I=(x_0^d, x_1^d, x_2^d, x_3^d)+x_0^{d-2}(x_0x_1,x_2x_3,
x_1^2,x_2^2,x_3^2)$ is a minimal monomial  Togliatti systems of forms of degree $d$ with $\mu(I)=2n+3=9$ and it is non-smooth.
\end{rem}

\begin{rem} \label{rem2} \rm
For $n=2$ and $6\le d\le 9$ one can check with the help of
Macaulay2 \cite{M} that there exist other examples of minimal monomial Togliatti systems $I=(x_{0}^d,x_{1}^d,x_{2}^d)+J\subset k[x_0,x_1,x_2]$ with $\mu(I)=7$. For seek of completeness we give the full list of possible $J$'s not included in Theorem \ref{mainthm1}:

 \vskip 2mm
 \noindent\underline{$d=6:$\;} {\small
 $(x_{0}^{5}x_{1}, x_{0}^{3}x_{2}^{3}, x_{0}^{2}x_{1}^{3}x_{2}, x_{1}^{5}x_{2}),
(x_{0}^{5}x_{2}, x_{0}^{3}x_{1}^{3}, x_{0}^{2}x_{1}^{2}x_{2}^{2}, x_{1}^{5}x_{2}),
(x_{0}^{3}x_{2}^{3}, x_{0}^{2}x_{1}^{4}, x_{0}^{2}x_{1}^{2}x_{2}^{2}, x_{1}^{5}x_{2}),
(x_{0}^{5}x_{2}, x_{0}^{3}x_{1}^{3}, $

\noindent$x_{0}x_{1}x_{2}^{4},x_{1}^{5}x_{2}),
(x_{0}^{4}x_{2}^{2}, x_{0}^{3}x_{1}^{3}, x_{0}^{2}x_{1}^{2}x_{2}^{2}, x_{1}^{4}x_{2}^{2}),
(x_{0}^{3}x_{2}^{3}, x_{0}^{2}x_{1}^{4}, x_{0}^{2}x_{1}^{2}x_{2}^{2}, x_{1}^{4}x_{2}^{2}),
(x_{0}^{4}x_{2}^{2}, x_{0}^{3}x_{1}^{3}, x_{0}x_{1}x_{2}^{4}, x_{1}^{4}x_{2}^{2}),
(x_{0}^{3}x_{1}^{3}, x_{0}^{3}x_{2}^{3},  $

\noindent$x_{0}^{2}x_{1}^{2}x_{2}^{2},x_{1}^{3}x_{2}^{3}),
x_{0}x_{1}(x_{0}^{4}, x_{0}^{2}x_{1}^{2},x_{0}x_{1}x_{2}^{2}, x_{1}^{4}),
x_{0}x_{1}(x_{0}^{3}x_{2}, x_{0}^{2}x_{1}^{2}, x_{0}x_{1}x_{2}^{2}, x_{1}^{3}x_{2}),
x_{0}x_{1}(x_{0}^{2}x_{1}^{2}, x_{0}^{2}x_{2}^{2}, x_{0}x_{1}x_{2}^{2}, x_{1}^{2}x_{2}^{2}),
$

\noindent$x_{0}x_{1}(x_{0}^{2}x_{1}^{2},x_{0}^{2}x_{2}^{2}, x_{0}x_{2}^{3}, x_{1}^{2}x_{2}^{2}),
x_{0}x_{1}(x_{0}^{4}, x_{0}x_{2}^{3}, x_{1}^{4}, x_{1}^{2}x_{2}^{2}),
x_{0}x_{1}(x_{0}^{4}, x_{0}^{2}x_{1}^{2}, x_{1}^{4}, x_{2}^{4}),
x_{0}x_{1}(x_{0}^{4}, x_{0}x_{1}x_{2}^{2}, x_{1}^{4}, x_{2}^{4}),
 $

\noindent$x_{0}x_{1}(x_{0}^{3}x_{2},x_{0}^{2}x_{1}^{2},x_{1}^{3}x_{2}, x_{2}^{4}),
x_{0}x_{2}(x_{0}^{3}x_{2}, x_{0}^{2}x_{2}^{2}, x_{0}x_{1}x_{2}^{2}, x_{1}^{4}),
x_{0}x_{2}(x_{0}^{2}x_{1}x_{2}, x_{0}^{2}x_{2}^{2}, x_{0}x_{1}x_{2}^{2}, x_{1}^{4}),
x_{0}x_{2}(x_{0}^{3}x_{2}, x_{0}x_{1}^{2}x_{2}, $

\noindent$x_{0}x_{1}x_{2}^{2}, x_{1}^{4}),
x_{0}x_{2}(\!x_{0}^{3}x_{2},x_{0}^{2}x_{1}x_{2}, x_{0}x_{2}^{3}, x_{1}^{4}),
x_{0}x_{2}(\!x_{0}^{3}x_{2}, x_{0}^{2}x_{2}^{2}, x_{0}x_{2}^{3}, x_{1}^{4}),
x_{0}x_{2}(\!x_{0}^{3}x_{2}, x_{0}x_{1}^{2}x_{2}, x_{0}x_{2}^{3}, x_{1}^{4}),
x_{0}x_{2}(\!x_{0}^{2}x_{1}^{2}, $

\noindent$x_{0}^{2}x_{2}^{2}, x_{0}x_{1}^{3}, x_{1}^{3}x_{2}),
x_{0}x_{2}(x_{0}^{2}x_{1}^{2}, x_{0}^{2}x_{2}^{2}, x_{0}x_{1}^{2}x_{2}, x_{1}^{3}x_{2}),
x_{0}x_{2}(x_{0}^{2}x_{2}^{2}, x_{0}x_{1}^{3}, x_{0}x_{1}^{2}x_{2}, x_{1}^{3}x_{2}),
x_{0}x_{2}(x_{0}^{2}x_{1}^{2}, x_{0}^{2}x_{2}^{2}, x_{0}x_{1}^{3}, x_{1}^{4}),
$

\noindent$x_{0}x_{2}(x_{0}^{2}x_{1}^{2}, x_{0}^{2}x_{2}^{2}, x_{1}^{4},x_{1}^{3}x_{2}),
x_{0}x_{2}(x_{0}^{2}x_{1}^{2},x_{0}^{2}x_{2}^{2}, x_{1}^{4}, x_{1}^{2}x_{2}^{2}),
x_{0}x_{2}(x_{0}^{3}x_{2}, x_{0}^{2}x_{1}x_{2}, x_{0}x_{1}^{2}x_{2}, x_{1}^{4}),
x_{0}x_{2}(x_{0}^{3}x_{2}, x_{0}^{2}x_{1}x_{2}, $

\noindent$x_{0}x_{1}x_{2}^{2}, x_{1}^{4}),
x_{0}(x_{0}x_{1}^{4}, x_{0}x_{1}x_{2}^{3},x_{0}x_{2}^{4}, x_{1}^{3}x_{2}^{2}),
x_{0}(x_{0}^{4}x_{2}, x_{0}^{2}x_{1}^{3}, x_{0}x_{1}^{2}x_{2}^{2}, x_{1}^{5}),
x_{0}(x_{0}^{4}x_{2}, x_{0}x_{1}x_{2}^{3}, x_{1}^{5}, x_{1}^{3}x_{2}^{2}),
x_{0}(x_{0}^{2}x_{2}^{3},  $

\noindent$x_{0}x_{1}^{4},x_{0}x_{1}^{2}x_{2}^{2}, x_{1}^{5}),
x_{0}(x_{0}^{4}x_{2}, x_{0}^{2}x_{1}x_{2}^{2},x_{1}^{5}, x_{1}^{2}x_{2}^{3}),
x_{0}(x_{0}^{2}x_{2}^{3}, x_{0}x_{1}^{4}, x_{0}x_{1}x_{2}^{3}, x_{1}^{3}x_{2}^{2}),
x_{0}(x_{0}^{4}x_{2}, x_{0}^{2}x_{2}^{3}, x_{0}x_{1}^{3}x_{2},x_{1}^{5}),
$

\noindent$x_{0}(x_{0}^{4}x_{2}, x_{0}^{2}x_{1}^{3}, x_{1}^{5}, x_{1}x_{2}^{4}),
x_{0}(x_{0}^{3}x_{2}^{2}, x_{0}^{2}x_{1}^{3}, x_{0}x_{2}^{4},x_{1}^{3}x_{2}^{2}),
x_{0}(x_{0}^{4}x_{2}, x_{0}x_{1}^{2}x_{2}^{2}, x_{1}^{5}, x_{1}x_{2}^{4}),
x_{0}(x_{0}^{2}x_{2}^{3}, x_{0}x_{1}^{4}, x_{0}x_{2}^{4}, x_{1}^{3}x_{2}^{2}),
$

\noindent $x_{0}(x_{0}^{2}x_{1}^{3}, x_{0}^{2}x_{2}^{3},
 x_{1}^{4}x_{2}, x_{1}x_{2}^{4}), x_{0}(x_{0}^{4}x_{1}, x_{0}^{2}x_{2}^{3}, x_{0}x_{1}^{3}x_{2}, x_{1}^{5}),
x_{0}(x_{0}^{2}x_{1}x_{2}^{2}, x_{0}x_{1}^{3}x_{2}, x_{1}^{5}, x_{2}^{5})$}

\vskip 2mm
\noindent \underline{$d=7:$\;} {\small
 $x_{0}x_{1}(x_{0}^{2}x_{2}^{3}, x_{0}x_{1}^{4}, x_{0}x_{1}^{2}x_{2}^{2}, x_{1}^{5}),
x_{0}x_{1}(x_{0}^{5}, x_{0}^{2}x_{1}^{2}x_{2}, x_{0}x_{1}x_{2}^{3}, x_{1}^{5}),
x_{0}x_{1}(x_{0}^{4}x_{1}, x_{0}^{3}x_{1}^{2}, x_{0}x_{2}^{4}, x_{1}^{3}x_{2}^{2}),
x_{0}x_{1}(x_{0}^{3}x_{1}^{2},$

\noindent$x_{0}^{2}x_{1}^{3}, x_{0}^{2}x_{2}^{3}, x_{1}^{2}x_{2}^{3}),
x_{0}x_{1}(x_{0}^{5}, x_{0}^{2}x_{1}^{2}x_{2}, x_{1}^{5}, x_{2}^{5}),
x_{0}x_{1}(x_{0}^{4}x_{2}, x_{0}x_{1}^{4}, x_{1}^{5}, x_{2}^{5}),
x_{0}x_{1}(x_{0}^{5}, x_{0}x_{1}x_{2}^{3}, x_{1}^{5}, x_{2}^{5}),
x_{0}x_{2}(x_{0}^{3}x_{2}^{2},$

\noindent$x_{0}^{2}x_{2}^{3},x_{0}x_{1}^{3}x_{2}, x_{1}^{5}),
x_{0}x_{2}(x_{0}^{4}x_{2}, x_{0}^{2}x_{1}x_{2}^{2}, x_{0}x_{2}^{4}, x_{1}^{5}),
x_{0}x_{2}(x_{0}^{4}x_{2}, x_{0}x_{1}^{3}x_{2}, x_{0}x_{2}^{4}, x_{1}^{5}),
x_{0}(x_{0}^{5}x_{2}, x_{0}^{2}x_{1}^{3}x_{2}, x_{0}x_{1}^{2}x_{2}^{3}, x_{1}^{6}),
$

\noindent$x_{0}(x_{0}x_{1}^{5},x_{0}x_{1}^{2}x_{2}^{3},x_{0}x_{2}^{5}, x_{1}^{4}x_{2}^{2}),
x_{0}(x_{0}^{5}x_{2}, x_{0}^{4}x_{1}^{2}, x_{0}^{2}x_{1}^{2}x_{2}^{2}, x_{1}^{3}x_{2}^{3}),
x_{0}(x_{0}^{4}x_{1}^{2}, x_{0}^{4}x_{2}^{2},
 x_{0}^{2}x_{1}^{2}x_{2}^{2}, x_{1}^{3}x_{2}^{3}),
 x_{0}(x_{0}^{3}x_{1}^{3}, x_{0}^{3}x_{2}^{3},$

\noindent$ x_{0}^{2}x_{1}^{2}x_{2}^{2}, x_{1}^{3}x_{2}^{3}),x_{0}(x_{0}^{4}x_{1}x_{2}, x_{0}^{2}x_{1}^{4}, x_{0}^{2}x_{2}^{4}, x_{1}^{3}x_{2}^{3}),
x_{0}(x_{0}^{2}x_{1}^{4}, x_{0}^{2}x_{1}^{2}x_{2}^{2}, x_{0}^{2}x_{2}^{4}, x_{1}^{3}x_{2}^{3}),
x_{0}(x_{0}^{4}x_{1}x_{2}, x_{0}x_{1}^{5}, x_{0}x_{2}^{5}, x_{1}^{3}x_{2}^{3}),
$

\noindent$x_{0}(x_{0}^{2}x_{1}^{2}x_{2}^{2}, x_{0}x_{1}^{5},x_{0}x_{2}^{5},x_{1}^{3}x_{2}^{3}),
x_{0}(x_{0}^{5}x_{2}, x_{0}^{2}x_{1}^{3}x_{2}, x_{1}^{6}, x_{1}x_{2}^{5}),
x_{0}(x_{0}^{5}x_{2}, x_{0}x_{1}^{2}x_{2}^{3}, x_{1}^{6}, x_{1}x_{2}^{5}),
x_{0}(x_{0}^{5}x_{2}, x_{0}^{4}x_{1}^{2}, x_{1}^{5}x_{2}, $

\noindent$x_{1}x_{2}^{5}),
x_{0}(x_{0}^{4}x_{1}^{2}, x_{0}^{4}x_{2}^{2},x_{1}^{5}x_{2},x_{1}x_{2}^{5}),
x_{0}(x_{0}^{3}x_{1}^{3}, x_{0}^{3}x_{2}^{3}, x_{1}^{5}x_{2}, x_{1}x_{2}^{5}),
x_{0}(x_{0}^{4}x_{1}x_{2}, x_{0}^{2}x_{1}^{2}x_{2}^{2}, x_{1}^{6}, x_{2}^{6}),
x_{0}(x_{0}^{4}x_{1}x_{2}, x_{1}^{6}, x_{1}^{3}x_{2}^{3}, $

\noindent$x_{2}^{6}),
x_{0}(\!x_{0}^{2}x_{1}^{2}x_{2}^{2}, x_{1}^{6},x_{1}^{3}x_{2}^{3}, x_{2}^{6}),x_{0}x_{1}x_{2}(\!x_{0}^{2}x_{1}^{2}, x_{0}^{2}x_{2}^{2}, x_{0}x_{1}^{3}, x_{1}^{4}),
x_{0}x_{1}x_{2}(\!x_{0}^{3}x_{2}, x_{0}^{2}x_{1}x_{2}, x_{0}x_{1}^{2}x_{2}, x_{1}^{4}),
x_{0}x_{1}x_{2}(\!x_{0}^{4}, x_{0}^{2}x_{1}^{2}, $

\noindent$x_{0}x_{1}x_{2}^{2}, x_{1}^{4}),
x_{0}x_{1}x_{2}(x_{0}^{3}x_{2},x_{0}^{2}x_{1}x_{2}, x_{0}x_{1}x_{2}^{2}, x_{1}^{4}),
x_{0}x_{1}x_{2}(x_{0}^{3}x_{2}, x_{0}^{2}x_{2}^{2}, x_{0}x_{1}x_{2}^{2}, x_{1}^{4}),
x_{0}x_{1}x_{2}(x_{0}^{2}x_{1}x_{2}, x_{0}^{2}x_{2}^{2}, x_{0}x_{1}x_{2}^{2}, $

\noindent$x_{1}^{4}),
x_{0}x_{1}x_{2}(x_{0}^{3}x_{2}, x_{0}x_{1}^{2}x_{2},x_{0}x_{1}x_{2}^{2}, x_{1}^{4}),
x_{0}x_{1}x_{2}(x_{0}^{3}x_{2}, x_{0}^{2}x_{1}x_{2}, x_{0}x_{2}^{3}, x_{1}^{4}),
x_{0}x_{1}x_{2}(x_{0}^{3}x_{2}, x_{0}^{2}x_{2}^{2}, x_{0}x_{2}^{3}, x_{1}^{4}),
$

\noindent$x_{0}x_{1}x_{2}(x_{0}^{3}x_{1}, x_{0}x_{1}^{3},x_{0}x_{2}^{3}, x_{1}^{4}),
x_{0}x_{1}x_{2}(x_{0}^{3}x_{2}, x_{0}x_{1}^{2}x_{2}, x_{0}x_{2}^{3}, x_{1}^{4}),
x_{0}x_{1}x_{2}(x_{0}^{2}x_{1}^{2}, x_{0}^{2}x_{2}^{2}, x_{0}x_{1}^{3}, x_{1}^{3}x_{2}),
x_{0}x_{1}x_{2}(x_{0}^{2}x_{1}^{2}, $

\noindent$x_{0}^{2}x_{2}^{2}, x_{0}x_{1}^{2}x_{2}, x_{1}^{3}x_{2}),
x_{0}x_{1}x_{2}(x_{0}^{2}x_{2}^{2}, x_{0}x_{1}^{3},x_{0}x_{1}^{2}x_{2}, x_{1}^{3}x_{2}),
x_{0}x_{1}x_{2}(x_{0}^{3}x_{2},x_{0}^{2}x_{1}^{2},x_{0}x_{1}x_{2}^{2}, x_{1}^{3}x_{2}),
x_{0}x_{1}x_{2}(x_{0}^{3}x_{1}, x_{0}^{2}x_{2}^{2}, $

\noindent$x_{0}x_{1}x_{2}^{2}, x_{1}^{3}x_{2}),
x_{0}x_{1}x_{2}(x_{0}^{2}x_{1}^{2}, x_{0}^{2}x_{2}^{2}, x_{1}^{4}, x_{1}^{3}x_{2}),x_{0}x_{1}x_{2}(x_{0}^{4}, x_{0}x_{1}^{3}, x_{0}x_{2}^{3}, x_{1}^{2}x_{2}^{2}),
x_{0}x_{1}x_{2}(x_{0}^{4},x_{0}x_{2}^{3}, x_{1}^{4}, x_{1}x_{2}^{3}).$}

\vskip 2mm

\noindent\underline{$d=8:$\;} {\small$x_{0}x_{1}(x_{0}^{4}x_{2}^{2}, x_{0}^{3}x_{1}^{3}, x_{0}x_{1}x_{2}^{4}, x_{1}^{4}x_{2}^{2}),
x_{0}x_{2}(x_{0}^{3}x_{2}^{3}, x_{0}^{2}x_{1}^{2}x_{2}^{2}, x_{0}x_{1}^{4}x_{2}, x_{1}^{6})$}

\vskip 2mm
 \noindent\underline{$d=9:$\;} {\small$x_{0}x_{1}x_{2}(x_{0}^{3}x_{2}^{3}, x_{0}^{2}x_{1}^{2}x_{2}^{2}, x_{0}x_{1}^{4}x_{2}, x_{1}^{6}),x_{0}x_{1}x_{2}(x_{0}^{3}x_{1}^{3}, x_{0}^{3}x_{2}^{3}, x_{0}^{2}x_{1}^{2}x_{2}^{2}, x_{1}^{3}x_{2}^{3}),
x_{0}x_{1}x_{2}(x_{0}^{6}, x_{0}^{2}x_{1}^{2}x_{2}^{2}, x_{1}^{6}, x_{2}^{6})$.}

\end{rem}

%%%%%%%%%%%%%%%%%%%%%%%%%%%%%%%%%%%%%%%%%%%%%%%%%%

\end{document}